\documentclass{article}
\usepackage{amssymb}
\usepackage{amsmath}
\usepackage{amsthm}

\numberwithin{equation}{section}

\theoremstyle{plain}

	\newtheorem{cor}{Corollary}[section]
	\newtheorem{prop}{Proposition}[section]
\theoremstyle{definition}

\theoremstyle{remark}
	\newtheorem{rem}{Remark}[section]

\newcommand{\MW}[6]{
W^{#1}\left(
\renewcommand{\arraystretch}{0.8}
\begin{array}{cccccccc}#2\end{array}
\renewcommand{\arraystretch}{1.0}
\Big|\,{#3};{#4};{#5}; {#6} \right)
}

\newcommand{\MMW}[4]{
W^{#1}\left(
\renewcommand{\arraystretch}{0.8}
\begin{array}{cccccccc}#2\end{array}
\renewcommand{\arraystretch}{1.0}
\Big|\,{#3};{#4}; \right. 
}

\begin{document}
\title{\bf Multiple basic hypergeometric transformation formulas arising from the 
balanced duality transformation}
\author{Yasushi KAJIHARA}
\date{}
\maketitle

\begin{abstract}
Some multiple hypergeometric transformation 
formulas arising from the balanced duality transformation formula are discussed 
through the symmetry. 
Derivations of some transformation formulas with different dimensions 
are given by taking certain limits of the balanced duality transformation.
By combining some of them, some transformation formulas for $A_n$ 
basic hypergeometric series is given. They include some generalizations 
of Watson, Sears and ${}_8 W_7$ transformations.
\end{abstract}





\begin{center}
\end{center}

{\allowdisplaybreaks

\section{Introduction}

\medskip
This paper can be considered as a continuation of our paper \cite{Kaji1}.
Namely, we discuss some multiple hypergeometric transformation 
formulas arising from the balanced duality transformation formula  
through the symmetry in this paper. 
We obtain  some transformation formulas with different dimensions 
 by taking certain limits of the balanced duality transformation.
By combining some of them, we give some transformation formulas for $A_n$ 
basic hypergeometric series. They include some generalizations 
of Watson, Sears and nonterminating and terminating ${}_8 W_7$ transformations.

The hypergeometric series ${}_{r+1} F_r$ is defined by 
\begin{equation}
{}_{r+1} F_r 
\left[ \
\begin{matrix}
a_0, & a_1, & a_2, & \cdots, & a_r \\
 & b_1, & b_2, & \cdots, & b_r \\  
\end{matrix}
\ \ ; z \
\right]
:= 
\sum_{k \in \Bbb N}
\frac
{[a_0, a_1, \cdots, a_r]_k}
{ k ! \ [b_1, \cdots, b_r]_k} \ 
z^k, 
\end{equation}
where $[c]_k = c (c+1) \cdots (c+ k -1)$ is Pochhammer symbol and
$[d_1, \cdots , d_r]_k = [d_1]_k \cdots [d_r]_k$.

The very well-poised hypergeometric series have nice properties such as the existence of 
various kinds of summation and transformation formulas which contains more parameters 
than other hypergeometric series and they have reciprocal structure (For precise see 
an excellent exposition by G.E.~Andrews \cite{AWP}).
 In \cite{B1}, W.N.Bailey derived the following transformation formula for terminating
balanced and very well-poised ${}_9 F_8$ series which is nowadays called the Bailey transformation 
formula:

\begin{eqnarray}\label{CBaileyT1}
&&
{}_{9} F_{8}
\left[ \ 
\begin{matrix}
a, & a/2 +1,& b, & c, & d,  \\ 
& a/2,  & 1 + a -b, & 1 + a-c, & 1 + a -d,  
\end{matrix}
\right.
\\ 
&&
\quad \quad \quad \quad \quad \quad \quad \quad
\quad \quad \quad \quad
\left.
\begin{matrix}
e, & f, & g, & -N \\
1 + a -e, & 1 +a - f, & 1 + a - g, & 1+ a +N 
\end{matrix} \ \ ; \ 1 \
\right] 
\nonumber
\\
&=&
\frac
{[1+ a, 1 + \lambda - e, 
1 + \lambda - f, 1 + \lambda - g]_N}
{[1+ \lambda, 1 + a - e, 
1 + a - f, 1 + a - g]_N}
\nonumber \\ 
& \times &
{}_{9} F_{8}
\left[ \ 
\begin{matrix}
\lambda, & \lambda/2 +1,&  \lambda + b - a, & \lambda + c - a, & \lambda +d -a,  \\ 
& \lambda /2,  & 1 + a -b, & 1 + a-c, & 1 + a -d,  
\end{matrix}
\right.
\nonumber 
\\ 
&&
\quad \quad \quad \quad \quad \quad \quad \quad
\quad \quad \quad \quad
\left.
\begin{matrix}
e, & f, & g, & -N \\
1 + \lambda -e, & 1 + \lambda - f, & 1 + \lambda - g, & 1+ \lambda +N 
\end{matrix} \ \ ; \ 1 \
\right], 
\nonumber 
\end{eqnarray} 
where $ \lambda = 1 + 2a - b- c- d$, and the parameters are subject to the restriction
(called as balancing condition) that
\begin{equation}\label{cbc}
2 + 3 a = b + c + d + e + f + g -N.  
\end{equation} 

 Among various hypergeometric transformation formulas, the Bailey 
transformation itself \eqref{CBaileyT1} and its special and limiting cases, including 
their basic and elliptic analogues, have many significant applications 
in various branches of mathematics and mathematical physics.
 
 On the other hand, Holman, Biedenharn and Louck \cite{H1}, \cite {HBL}
introduced a class of multiple generalization of hypergeometric series
in need of the explicit expressions of the Clebsch-Gordan 
coefficients of the tensor product of certain irreducible 
representations of the Lie group $SU(n+1)$. 
In the series of papers, S.Milne \cite{MilneG} introduced its basic analogue and 
investigated further. In the course of works (see the expository paper by Milne 
\cite{MilneNagoya} and references therein),
 he and his collaborators succeeded to obtain some multiple hypergeometric transformation 
and summation formulas by using a certain rational function identity
which is nowadays referred as Milne`s fundamental lemma. After that, 
many methods of the derivation of multiple hypergeometric identities 
have been worked out such as ingenious uses of certain matrix techniques
by Krattenthaler, Schlosser, Milne, Lilly and Newcomb, see 
\cite{BS}, \cite{Milne3},  \cite{LM1}, \cite{MilNew1}. 

 In the previous work \cite{Kaji1}, we derived a certain generalization
of the Euler transformation formula for multiple (basic) hypergeometric
series with different dimensions by using the techniques in the 
theory of Macdonald polynomials and Macdonald`s $q$-difference operators. 
By interpreting our multiple Euler transformation formula as 
the generating series, we further obtained several types of multiple 
hypergeometric summations and transformations. 

Among these, 
we consider the {\it (balanced)  duality transformation formula } \eqref{BDT1} which 
generalize the following ${}_9 F_8$ transformation 
(see Bailey`s book \cite{BB}):
\begin{eqnarray}\label{Cmn1BDT1}
&&
{}_{9} F_{8}
\left[ \ 
\begin{matrix}
a, & a/2 +1,& b, & c, & d,  \\ 
& a/2,  & 1 + a -b, & 1 + a-c, & 1 + a -d,  
\end{matrix}
\right.
\\ 
&&
\quad \quad \quad \quad \quad \quad \quad \quad
\quad \quad \quad \quad
\left.
\begin{matrix}
e, & f, & g, & -N \\
1 + a -e, & 1 +a - f, & 1 + a - g, & 1+ a +N 
\end{matrix} \ \ ; \ 1 \
\right] 
\nonumber
\\
&=&
\frac
{[1+ a, 1 + a - b-c, 1 + a -b-d, 1+a -b -e, 1 + a - b -f, g]_N}
{[1+ a- b, 1 + a -c, 1 + a -d, 1 + a - e, 
1 + a - f, g- b]_N}
\nonumber \\ 
& \times &
{}_{9} F_{8}
\left[ \ 
\begin{matrix}
 b - g - N, & (b - g - N)/2 +1,&  b, & 1 + a - c- g, & 1 + a -d -g,  \\ 
& (b - g -N)/2,  & 1 - g - N, &  b + c - a- N, & b + d - a- N,  
\end{matrix}
\right.
\nonumber 
\\ 
&&
\quad \quad \quad \quad \quad \quad \quad \quad
\quad \quad 
\left.
\begin{matrix}
1 + a - e- g, &  1+ a - f - g, & b - a -N, & -N \\
b + e - a - N, & b + f - a -N, & 1 + a - g, & 1+ b - g 
\end{matrix} \ \ ; \ 1 \
\right], 
\nonumber 
\end{eqnarray} 
with the balancing condition \eqref{cbc}, as the formula of particular importance.
\eqref{Cmn1BDT1} is different from \eqref{CBaileyT1}, but \eqref{Cmn1BDT1} can be obtained 
by duplicating \eqref{CBaileyT1}. 
In the joint work with M.Noumi \cite{KajiNou}, we derived an 
analogous formula ((3.17) of \cite{KajiNou}) for elliptic hypergeometric series introduced by
Frenkel and Turaev \cite{FT} by starting from the Frobenius determinant and
proposed the notion of { (balanced) duality transformation}
there.

 In this paper, we present an alternative approach by 
 starting from the balanced duality transformation formula
for multiple  hypergeometric series of type $A$.
In Section 3, we present some transformation formula for basic hypergeometric 
series of type $A$ with different dimension. They include most of results 
in our previous work \cite{Kaji1}. What is remarkable is that from the balanced 
duality transformation formula, one can obtain multiple Euler transformation
formula itself.

 By iterating twice  a special case of our Sears 
transformation formula (see section 7 in \cite{Kaji1}),
we verified an $A_n$ Sears transformation formula in \cite{KajiS}. 
Later by the same idea as above, we obtained in \cite{KajiNou} two types
of $A_n$ Bailey transformation formulas: one of which is previously
known by  Milne and Newcomb \cite{MilNew1} in  basic case and Rosengren
\cite{RoseE} in elliptic case and another has appeared to be new in both cases. 
 In Section 4, we further employ this idea to obtain several $A_n$ basic hypergeometric 
transformations which generalize Watson, Sears and nonterminating ${}_8 W_7$ 
transformations. They includes known ones due to Milne and his collaborators.
We will see here that our class of multiple hypergeometric transformations 
shed a light to the structure of some $A_n$ hypergeometric transformation formulas.


\section{Preliminaries}
 
\medskip 

In this section, we give some notations for multiple basic hypergeometric 
series and present the balanced duality transformation formula. 
We basically refer 
the notations of $q$-series and basic hypergeometric series
from the book by Gasper and Rahman \cite{GR1}. Throughout of this paper, 
we assume that $q$ is a complex number under the condition $0 < |q|< 1$.   
We define $q-$shifted factorial as 
\begin{equation}\label{ShiftFact}
(a)_\infty := (a;q)_\infty =\prod_{n \in \Bbb N}(1 - a q^n), \quad 
(a)_k := (a;q)_k = \frac{(a)_\infty}{(a q^k)_\infty} 
\quad \mbox{for} 
\ k \in \Bbb C.
\end{equation}
where, unless stated otherwise, we omit the basis $q$.
In this paper we employ the notation as
\begin{equation}
(a_1)_k \cdot (a_2)_k \cdots (a_n)_k
=(a_1, a_2, \cdots , a_n)_k.
\end{equation}

The basic hypergeometric series ${}_{r+1} \phi_{r}$ is defined by 
\begin{eqnarray}\label{BHSdef}
&&
{}_{r+1} \phi_r 
\left[
	\begin{matrix}
		a_1, a_2, \ldots  a_{r+1}\\
 			c_1,  \ldots  c_r
\end{matrix}
;q; u
\right] =
{}_{r+1} \phi_r 
\left[
	\begin{matrix}
		a_0, \{ a_i \}_{r}\\
 			\{c_i  \}_r
\end{matrix}
;q; u
\right]
\\
&&
\quad \quad \quad \quad \quad \quad \quad
:= \sum_{k \in \Bbb N} 
\frac{(a_0, a_1, \ldots ,a_{r+1})_k }
{( c_1, \ldots , c_r, q )_k} u^k   
\nonumber
\end{eqnarray}
with $r+1$ numerator parameters $ a_0, a_1, \cdots , a_{r+1}$ and 
$r$ denominator parameters $c_1, \cdots , c_r$.
We call ${}_{r+1} \phi_r$ series $k$-balanced if 
$ q^k a_1 a_2 \cdots a_{r+1} = c_1 \cdots c_r$ and $u=q$: a 
$1$-balanced series is called balanced (or Saalsch\"{u}tzian).
An ${}_{r+1} \phi_{r}$ series is called well-poised if 
$a_0 q = a_1 c_1 = \cdots = a_r c_r$. In addition, if  
$a_1 = q \sqrt{a_0}$ and $a_2 = - q \sqrt{a_0}$,  then 
the ${}_{r+1} \phi_r $ is called very well-poised.
We denote the very well-poised basic 
hypergeometric series ${}_{r+3} \phi_{r+2} $ as 
${}_{r+3} W_{r+2}$ series defined by the following:

\begin{eqnarray}\label{DefVWP}
&&{}_{r+3} W_{r+2}
	\left[
		 s;  \{a_i \}_r; 
		 q; u
	\right]
:=
{}_{r+3} \phi_{r+2}
	\left[
		\begin{matrix}
			& s, & q \sqrt{s},&  - q \sqrt{s},& \{ a_i \}_r \\
			& & \sqrt{s},& - \sqrt{s},& \{ s q / a_i \}_r 
		\end{matrix};
		q; u
	\right]
\\
&&
\quad \quad \quad \quad \quad \quad 
\ = \ 
\sum_{k \in {\Bbb N}} 
\frac{1 - s q^{2 k} }{1 - s}
\frac{( s, a_1 \cdots ,a_r)_k}
{(q, s q / a_1,  \cdots , s q /a_r)_k}
u^k.
\nonumber
\end{eqnarray}
Furthermore, all of the very well-poised basic hypergeometric series 
${}_{r+3} W_{r+2}$ in this paper, the parameters $s, a_1, \cdots a_r$ and 
the argument satisfy the very-well-poised-balancing condition
\begin{equation}\label{vwpb}
a_1 \cdots a_r u = \left( \pm (s q)^{\frac{1}{2}} \right)^{r-1} 
\end{equation}
with either the plus and minus sign. We call a ${}_{r+3} W_{r+2}$ series very-well-poised-balanced
if \eqref{vwpb} holds. Note that a very-well-poised-balanced ${}_{r+3} W_{r+2}$ series is 
($1$-)balanced if
\begin{equation}
a_1 \cdots a_r = s^{\frac{r-1}{2}} q^{\frac{r-3}{2}}
\end{equation}
and $u=q$.

Now, we note the conventions for naming series as $A_n$ basic hypergeometric series 
(or basic hypergeometric series in $SU(n+1)$).
Let $\gamma= (\gamma_1, \cdots , \gamma_n) \in \mathbb{N}^n$ be a multi-index.
We denote 
\begin{equation}
\Delta(x) := \prod_{1 \le i < j \le n}
(x_i - x_j)
\quad  \mbox{and}
\quad
\Delta(x q^\gamma) := \prod_{1 \le i < j \le n}
(x_i q^{\gamma_i} - x_j q^{\gamma_j}),  
\end{equation}
as the Vandermonde determinant for the sets of variables $x=(x_1, \cdots , x_n)$
and $x q^\gamma = (x_1 q^{\gamma_1}, \cdots ,  x_n q^{\gamma_n})$  respectively.
In this paper we refer multiple series of the form
\begin{equation}\label{DefAnBHS}
\sum_{\gamma \in {\Bbb N}^n}
\frac{\Delta (x q^{\gamma})}{\Delta (x)}
\
H(\gamma)
\end{equation}
which reduce to basic hypergeometric series ${}_{r+1} \phi_r$ for a nonnegative 
integer $r$ when $n=1$ and symmetric with respect to the subscript $1 \le i \le n$
as $A_n$ basic hypergeometric series.
We call such a series balanced if it reduces to a balanced series when $n=1$.
Very well-poised and so on are defined similarly.
The subscript $n$ in the label $A_n$ attached to the series is the dimension of 
the multiple series \eqref{DefAnBHS}.

In our previous work \cite{Kaji1},
we derived a hypergeometric transformation formula for multiple basic hypergeometric series
of type $A$ generalizing the following transformation for terminating balanced ${}_{10} W_9$ series
in the one dimensional case:
\begin{eqnarray}\label{m1n1BDT1}
&&
{}_{10} W_{9} 
\left[
	\begin{matrix}
		a; b, c, d, e, f,  \mu f q^N, q^{-N}
	\end{matrix};
q; q 
\right]
=
\frac
{(\mu b f/ a, \mu c f / a, \mu d f / a, \mu e f / a, 
a q, f)_N }
{(a q / b, a q / c, a q / d, a q/e, \mu q, \mu f / a)_N}
\\
&& 
\quad \quad \quad \quad \quad \quad  \quad \quad \quad 
\times
{}_{10} W_{9} 
\left[
	\begin{matrix}
		\mu ; a q / b f,  a q / c f,  a q / d f,  a q / e f,  
        \mu f / a, \mu f q^{N}, q^{-N}
	\end{matrix};
q;  q
\right],
\nonumber
\end{eqnarray}
where $ \mu = a^3 q^2 /bcde f^2$. Note that \eqref{m1n1BDT1} is a basic analogue of 
\eqref{Cmn1BDT1}. The transformation \eqref{m1n1BDT1}  can be obtained by iterating 
 the  Bailey transformation formula for ${}_{10} W_9$ series 

\begin{eqnarray}\label{BaileyT1}
&&
{}_{10} W_{9} 
\left[
		a;b, c, d, e, f, 
		\lambda a q^{N+1}/e f, q^{-N}
; q; q
\right]
=
\frac{(a q, a q /e f, \lambda q /e, \lambda q /f)_N }
{(a q/ e, a q /f, \lambda q, \lambda q /e f)_N }
\\
&&
\quad \quad \quad 
\times
{}_{10} W_{9} 
\left[
		\lambda; \lambda b /a,
		\lambda c/ a, \lambda d / a , e, f, 
		\lambda a q^{N+1}/e f, q^{-N}
; q; q
\right]
\quad \quad \quad
( \lambda = a^2 q^{} /b c d),
\nonumber
\end{eqnarray}
twice.
Note also that \eqref{BaileyT1} is a basic analogue of \eqref{CBaileyT1}.

To simplify the expressions for multiple very well-poised series,
we introduce the notation of $W^{n,m}$ series as the following:

\begin{eqnarray}\label{DefWSer}
&&
\MW{n, m}{\{a_i\}_n \\ \{x_i\}_n}
{s}{ \{u_k \}_{m}}{\{v_k \}_{m}}
{q; z}\\ 
&&
\quad
:=
\sum_{\gamma \in \mathbb{N}^n}{}\ z^{|\gamma|}
\prod_{1\le i<j\le n}{}\,
\frac{\Delta(x q^\gamma )}{\Delta(x)}\ 
\prod_{1 \le i \le n}
\frac{1 - sq^{|\gamma|+\gamma_i}x_i}
{1 - s x_i}
\nonumber
\\ 
&&
\quad \quad
\times
\prod_{1\le j \le n}
\frac{(s x_j)_{|\gamma|}}{( (s q /a_j) x_j) _{|\gamma|}}
\left(
\prod_{1 \le i \le n}
\frac{(a_j x_i / x_j)_{\gamma_i}}{(q x_i / x_j )_{\gamma_i}}
\right)
\nonumber
\\
&&
\quad \quad \quad
\times
\prod_{1 \le k \le m}
\frac{(v_k)_{|\gamma|}}{( s  q /u_k)_{|\gamma|}}
\left(
\prod_{1 \le i \le n}
\frac{(u_k  x_i)_{\gamma_i}}{( ( s q/ v_k )x_i)_{\gamma_i}}
\right),\nonumber
\end{eqnarray}
where 
$|\gamma| = \gamma_1 + \cdots + \gamma_n$ is the length of a multi-index $\gamma$.

\medskip

The very starting point of all the discussions of the present paper 
is the {\it balanced duality transformation formula} (Corollary 6.3 of \cite {Kaji1}
with a different notation) between the $W^{n,m+2}$ series ($A_n$ ${}_{2m+8} W_{2m+7}$ series)
and $W^{m, n+2}$ series ($A_m$ ${}_{2n+8} W_{2n+7}$ series):
\begin{eqnarray}\label{BDT1}
&&
\MW{n,m+2}
{ \{ b_i\}_n \\ \{ x_i\}_n}
{a}{\{ c_k y_k \}_m, d, e}
{\{f y_k^{-1}\}_m, \mu f q^N, q^{-N}}
{q; q}
\\
&&
\quad
=
\frac{(\mu d f / a, \mu e f / a)_N}
{(a q / d, a q / e)_N}
\prod_{ 1 \le k \le m}
\frac{((\mu c_k f / a) y_k, f y_k^{-1} )_N}
{(\mu q y_k, (a q /c_k) y_k^{-1} )_N}
\prod_{ 1 \le i \le n}
\frac{(a q x_i, ( \mu b_i f / a ) x_i^{-1} )_N}
{((a q / b_i) x_i, (\mu f / a) x_i^{-1} )_N}
\nonumber\\
&&
\quad \quad 
\times
\MMW{m,n+2}
{\{ a q / c_k f \}_m \\  \{y_k\}_m}
{\mu}{\{(a q / b_i f) x_i \}_n, a q/ d f, a q/ e f}
\nonumber \\
&&
\quad \quad \quad \quad \quad
\quad \quad \quad \quad \quad
\quad \quad \quad \quad \quad
\quad \quad \quad \quad \quad
\quad \quad
{\{ (\mu f / a) x_i^{-1} \}_n, \mu f q^N, q^{-N}};
{q; q}
\Big),\nonumber
\end{eqnarray}
where $\mu = a^{m+2} q^{m+1} /BCde f^{m+1}$.
Here we denote  $B= b_1 \cdots b_n$ and $ C = c_1 \cdots c_m$. 
In this paper, we frequently use such notations.
 
In the case when $m=1$ and $y_1= 1$, \eqref{BDT1} reduces to
\begin{eqnarray}\label{m1BDT1}
&&
\MW{n,3}
{ \{ b_i\}_n \\ \{ x_i\}_n}
{a}{c, d, e} 
{f, \mu f q^N, q^{-N}}
{q; q} \\ 
&&
\quad 
=
\frac{(\mu c f /a, \mu d f / a, \mu e f / a, f )_N}
{(a q/ c, a q / d, a q / e, \mu q )_N }
\prod_{ 1 \le i \le n}
\frac{(a q x_i, ( \mu b_i f / a ) x_i^{-1} )_N}
{((a q / b_i) x_i, ( \mu f / a) x_i^{-1} )_N}
\nonumber
\\
&&
\quad \quad
\times
{}_{2 n + 8}W_{2 n + 7}
\left[
\mu; \{ (a q / b_i f) x_i \}_n, a q / c f, a q / d f, a q / e f, 
\right.
\nonumber
\\
&& \qquad \qquad \qquad \qquad \quad  \quad  
\left.
\{ (\mu f / a) x_n / x_i\}_n, \mu f q^N, q^{-N}; q; q
\right],
\quad
(\mu = a^{3} q^{2} / B c d e f^{2}). 
\nonumber
\end{eqnarray}
Note that $m=n=1$ and $x_1 = y_1 = 1$ case of  the balanced duality transformation formula 
\eqref{BDT1} is terminating balanced ${}_{10} W_9$ transformation \eqref{m1n1BDT1}.

In \cite{Kaji1}, \eqref{BDT1} was obtained by taking the coefficients of $u^N$
 in both sides of  "$0$-balanced"  case of the multiple Euler transformation formula 
for multiple basic hypergeometric series of type $A$ with different
dimensions (Theorem 1.1 of \cite{Kaji1})

\begin{eqnarray}\label{ETG}
&&
\sum_{\gamma \in {\Bbb N}^n}
u^{|\gamma|} 
\frac{\Delta (x q^{\gamma})}{\Delta (x)}
\prod_{1 \le i, j \le n}
\frac{(a_j x_i / x_j)_{\gamma_i}} {(q x_i / x_j)_{\gamma_i}}
\prod_{1 \le i \le n, 1 \le k \le m}
\frac{(b_k x_i y_k)_{\gamma_i}}{(c x_i y_k)_{\gamma_i}}
\\
&& 
\quad 
=
\
\frac{(A B  u/ c^m)_\infty}
{(u)_\infty}
\
\sum_{\delta \in {\Bbb N}^m}
 \left( \frac{A B  u}{c^m} \right)^{|\delta|} 
\frac{\Delta (y q^{\delta})}{\Delta (y)}
\nonumber\\
&& 
\quad \quad 
\times
\prod_{1 \le k, l \le m}
\frac{((c / b_l^{}) y_k / y_l)_{\delta_k}}{(q y_k / y_l)_{\delta_k}}
\prod_{1 \le i \le n, 1 \le k \le m}
\frac{((c /a_i^{}) x_i y_k )_{\delta_k}}
{(c x_i y_k)_{\delta_k}}. 
\nonumber
\end{eqnarray}
From this point of view, we can state the balanced duality transformation \eqref{BDT1}
in more general form:
(Proposition 6.2 in \cite{Kaji1})
\begin{eqnarray}\label{HBDT1}
&&
\sum_{\gamma \in {\Bbb N}^n, |\gamma| = N}
\frac{\Delta (x q^{\gamma})}{\Delta (x)}
\prod_{1 \le i, j \le n}
\frac{(a_j x_i / x_j)_{\gamma_i}} {(q x_i / x_j)_{\gamma_i}}
\prod_{1 \le i \le n, 1 \le k \le m}
\frac{(b_k x_i y_k)_{\gamma_i}}{(c x_i y_k )_{\gamma_i}}
\\
&& 
\quad 
= \ 
\sum_{\delta \in {\Bbb N}^m, |\delta| = N}
\frac{\Delta (y q^{\delta})}{\Delta (y)}
\prod_{1 \le k, l \le m}
\frac{((c / b_l^{}) y_k / y_l)_{\delta_k}}{(q y_k / y_l)_{\delta_k}}
\prod_{1 \le i \le n, 1 \le k \le m}
\frac{((c /a_i^{}) x_i y_k )_{\delta_k}}
{(c x_i y_k )_{\delta_k}}, 
\nonumber
\end{eqnarray}
when $ AB = c^m$.

The balanced duality transformation formula \eqref{BDT1} corresponds to the case 
when $ m,n \ge 2$ of \eqref{HBDT1} 
by a rearrangement of parameters in the multiple basic hypergeometric series.
Note that, in the case when $m=n=1$, \eqref{HBDT1} becomes tautological.
We shall also remark that the remaining case ($m=1,  n \ge 2$) corresponds to 
the $A_n$ Jackson summation formula for terminating balanced $W^{n,2}$ series
\begin{eqnarray}\label{JS1}
&&  
\MW{n,2}
{ \{ b_i\}_n \\ \{ x_i\}_n}
{a}{ c, e}
{d, q^{-N}}
{q; q} \\
&& 
\quad \quad \quad \quad \quad \quad 
\ = \
\frac{(a q / B c, a q / c d)_N}
{(a q / B c d, a q / c)_N}
\prod_{1 \le i \le n}
\frac{((a q /b_i d) x_i, a q x_i)_N}
{((a q /b_i) x_i, (a q/ d) x_i)_N}
\nonumber
\end{eqnarray}
provided $ a^2 q^{N+1} = B cde$, 
which is originally due to S.Milne \cite{Milne2} in a different notation. 
For these facts, the informed readers might see \cite{Kaji1} for ordinary and basic case 
and Noumi and the author \cite{KajiNou} for elliptic case.

\begin{rem}
In the case when $n=1$ and $x_1= 1$, \eqref{JS1} reduces to the Jackson summation formula for terminating
balanced ${}_8 W_7$ series:
\begin{eqnarray}\label{JacksonSum}
{}_8 W_7 
[
{a}; b, { c, d, e}
, { q^{-N}};
{q; q}
]
&=&
\frac{(a q, a q / b c, a q / b d, a q / c d)_N}
{(a q / b c d, a q / b, a q / c, a q / d)_N}
, \quad 
(a^2 q^{N+1} = b c d e).
\end{eqnarray}
We also mention that \eqref{JS1} can be obtained by letting $aq = ef$ in \eqref{m1BDT1} and by relabeling 
the parameters. 
\end{rem}


\section{Limit cases}

In this section, we shall show that several  transformation
formulas with different dimension can be obtained from the balanced 
duality transformation formula
\eqref{BDT1}
by taking certain limits.  
We see that most of principal transformation formulas 
in \cite{Kaji1} which have been obtained by taking a certain 
coefficient in the multiple Euler transformation formula \eqref{ETG}
can be recovered and we find some new transformation formulas.
Furthermore we show that \eqref{ETG} itself can be acquired in this 
manner. 
 
 In addition, we shall write down the cases when the dimension of
summand in each side of the transformation is one with particular 
attention.
We consider them as particularly significant ones
since they have an extra symmetry.
We will explore some $A_n$ hypergeometric transformations
by using some of them in the next section.

\subsection{(Non-balanced) Duality transformation formula and its inverse}

\medskip

\noindent
{\bf {(Non-balanced) Duality transformation formula}}

\medskip

\begin{prop}
\begin{eqnarray}\label{DT1}
&&
\MW{n, m+1}{\{b_i \}_n \\ \{x_i\}_n}{a}{c, \{d_k y_k  \}_m}
{q^{-N}, \{e  y_k^{-1} \}_m}
{q; \frac{a^{m+1} q^{N + m + 1} }{B c D e^m}}
\\
&&
\quad 
=
\
\frac{(a^{m+1} q^{m+1} /B c D e^m)_N}
{(a q /c)_N}
\prod_{ 1 \le i \le n}
\frac
{(a q x_i)_N}{((a q / b_i)x_i)_N}
\prod_{ 1 \le k \le m}
\frac
{(e y_k^{-1})_N}{((a q/ d_i) y_k^{-1} )_N}
\nonumber\\
&&
\quad \quad 
\times 
\sum_{\delta \in {\Bbb N}^m}
q^{|\gamma|}
\frac{\Delta (y q^{\delta})}{\Delta (y)}
\frac
{(q^{-N})_{|\gamma|}}
{(a^{m+1 } q^{m+1} /B c D e^m)_{|\gamma|}}
\prod_{1 \le k,l \le m}
\frac{((a q / d_l e ) y_k / y_l)_{\delta_k}} {(q y_k / y_l)_{\delta_k}}
\nonumber 
\\
&&
\quad \quad \quad
\times
\prod_{1 \le i \le n, 1 \le k \le m}
\frac{( (a q /b_i e) x_i y_k )_{\delta_k}}
{((a q / e) x_i y_k )_{\delta_k}}
\prod_{1 \le k \le m}
\frac{( (a q /c e) y_k)_{\delta_k}}
{(q^{1-N} e^{-1} y_k)_{\delta_k}},
\nonumber
\end{eqnarray}
where $B = b_1 \cdots b_m$ and $D = d_1 \cdots d_n$.
\end{prop}

\begin{proof}
Take the limit $e \to \infty$ in \eqref{BDT1}. By rearranging 
the parameter as $f \to e$, we arrive at the desired identity.
\end{proof} 

\begin{rem}
The transformation formula \eqref{DT1} has already appeared in Section 6.1 of our previous work
\cite{Kaji1} with a different notation.  Later Rosengren has also 
obtained in \cite{RoseKM} by using his reduction formula of 
Karlsson-Minton type. In contract to that we see here, 
as is mentioned in \cite{RoseKM}, the balanced duality transformation 
formula \eqref{BDT1} can also be considered as 
a special case of \eqref{DT1}.
Indeed, in \cite{Kaji1} though we have started the derivation of \eqref{DT1}
 from the multiple Euler transformation \eqref{ETG} in general case, \eqref{BDT1} 
has been obtained from the "zero"-balanced case of \eqref{ETG}. 

 In the case when $m=n=1$ and $x_1 = y_1 = 1$, \eqref{DT1} reduces to
\begin{eqnarray}\label{nm1DT1}
&&
{}_{8} W_{7} 
\left[	
a; b, c, d, e, q^{-N}; q; \frac{a^2 q^{N+2}}{b c d e}
\right]
\\
&& \quad \quad \quad 
=
\frac{(a^2 q^2 / b c d e, e, a q)_N}
{(a q / b, a q / c, a q / d)_N}
{}_{4} \phi_{3} 
\left[
	\begin{matrix}
		q^{-N}, a q / b e, a q / c e, a q / d e \\
		q^{1-N} /e, a^2 q^2 / b c d e, a q / e
	\end{matrix};
q; q
\right].
\nonumber
\end{eqnarray}
\end{rem}

 In \cite{Kaji1}, \eqref{DT1} is referred as Watson type transformation formula.
But, hereafter we shall propose to call \eqref{DT1} (non balanced) duality transformation formula. 

In the case when $m=1$ and $y_1$, \eqref{DT1} reduces to 

\begin{eqnarray}\label{m1DT1}
&&
\MW{n, 2}{\{b_i \}_n \\ \{x_i\}_n}{a}{c, d}
{q^{-N}, e }
{q; \frac{a^{2} q^{N + 2} }{B c d e}}
=
\frac
{( a^{2} q^{2} / B c d e, e)_N}
{(a q /c, a q / d)_N }
\prod_{ 1 \le i \le n}
\frac
{(a q x_i)_N }
{((a q / b_i) x_i)_N }
\nonumber\\
&& \quad \quad \quad 
\times 
{}_{n+3} \phi_{n+2} 
\left[
	\begin{matrix}
		q^{-N}, \{(a q /b_i e) x_i \}_{n}, 
	        a q / c e, a q / d e \\
		q^{1-N} /e, \{(a q /e) x_i \}_{n}, 
	        a^2 q^2 / B c d e, 
	\end{matrix};
q; q
\right].
\end{eqnarray} 

In the case of $n=1$ and $x_1 =1$, \eqref{DT1} reduces to
\begin{eqnarray}\label{n1DT1}
&&{}_{2 m + 6} W_{2 m + 5}
\left[	
a; b, \{c_k y_k \}_{m}, 
d, \{e y_k^{-1}  \}_{m}, q^{-N}; q; 
\frac{a^{m+1} q^{N+m+1}}{b C d e^m}
\right]
\\
&& 
\quad =
\frac
{( a^{m+1} q^{m+1} / b C d e^{m}, a q)_N}
{(a q / b, a q / d)_N }
\prod_{ 1 \le k \le m}
\frac
{(e y_k^{-1} )_N}
{((a q /c_k) y_k^{-1} )_N}
\nonumber\\
&& 
\quad \quad 
\times
\sum_{\delta \in {\Bbb N}^{m} }
q^{|\delta|}
\frac{\Delta ({y} q^{\delta})}{\Delta ({y})}
\prod_{1 \le k, l \le m}
\frac{(( a q / c_l e) y_k / y_l)_{\delta_k}}{(q y_k / y_l)_{\delta_k}}
\nonumber\\
&& 
\quad  \quad \quad \times
\frac{(q^{-N})_{|\delta|}}
{( a^{m+1} q^{m+1} / b C d e^{m})_{|\delta|}}
\prod_{1 \le k \le m}
\frac{((a q /b e) y_k, (a q / d e ) y_k)_{\delta_k}}
{(( a q / e) y_k, ( q^{1-N} / e) y_k)_{\delta_k}}.
\nonumber
\end{eqnarray}

Further, by letting $aq = de$ in \eqref{m1DT1}, we obtain Milne`s $A_n$ generalization 
of Rogers` summation formula for terminating very well-poised ${}_6 W_5$ series:

\begin{eqnarray}\label{RS1}
\MW{n, 1}{\{b_i \}_n \\ \{x_i\}_n}{a}{c}
{q^{-N}}
{q; \frac{a^{} q^{N + 1} }{B c}}
&=&
\frac
{( a^{} q^{} / B c)_N}
{(a q /c)_N }
\prod_{ 1 \le i \le n}
\frac
{(a q x_i)_N }
{((a q / b_i) x_i)_N }.
\end{eqnarray}
\begin{rem}
In the case when $n=1$ and $x_1= 1$, \eqref{RS1} reduces to
the Rogers` summation formula ( (2.4.2) in \cite{GR1} ): 
\begin{eqnarray}\label{RogersSum1}
{}_{6} W_{5}
\left[ a; b, c, q^{-N}; q; \frac{a q^{N+1}}{b c}
\right]
&=&
\frac{(a q, a q / b c)_N}
{(a q / b, a q /c )_N}.
\end{eqnarray}
\end{rem}

\medskip

\noindent
{\bf {\large The inverse of the duality transformation}}

\medskip

\begin{prop}
\begin{eqnarray}\label{IDT1}
&&
\sum_{\gamma \in {\Bbb N}^{n}}
q^{|\gamma|} 
\frac{\Delta({x} q^{\gamma})}{\Delta({x})}
\prod_{1 \le i, j \le n}
\frac{(a_j x_i / x_j)_{\gamma_i}} 
{(q x_i / x_j)_{\gamma_i}}
\prod_{1 \le i \le n, 1 \le k \le m}
\frac
{(b_k x_i y_k)_{\gamma_i} }
{(e x_i y_k )_{\gamma_i} }
\\
&& \quad \times
\frac{(q^{-N})_{|\gamma|}}
{(d)_{|\gamma|}}
\prod_{1 \le i\le n}
\frac{(c x_i)_{\gamma_i}} 
{( (AB c q^{1-N}/ d e^m) x_i)_{\gamma_i}}
\nonumber 
\\
&&
\quad \quad 
=
\frac{(d e^m /A B)_N}
{(d)_N}
\prod_{ 1 \le k \le m}
\frac
{((d e^m b_k / A B c) y_k )_N}
{((d e^{m+1}/ A B c) y_k )_N}
\prod_{ 1 \le i \le n}
\frac
{((d e^m a_i / A B c) x_i^{-1} )_N}
{((d e^m/ A B c) x_i^{-1} )_N}
\nonumber\\
&& \quad \quad \quad 
\times
\MMW{m, n+1}{\{e/ b_k \}_m \\ \{y_k\}_m}{d e^{m+1} q^{-1}/ A B c}
{ e/ c, \{ (e / a_i) x_i  \}_n}
\nonumber
\\
&&
\quad \quad \quad \quad \quad 
\quad \quad \quad \quad \quad 
\quad \quad \quad \quad \quad 
\quad \quad \quad  
{q^{-N}, \{ (d e^m / A B c) x_i^{-1} \}_n};
{q; d q^N}
\big).
\nonumber
\end{eqnarray}
\end{prop}

\begin{proof} 
Substitute the parameters $e$ and $f$ as $aq/e$ and $aq/f$ respectively in \eqref{BDT1}.
Then take the limit $a \to \infty$. Finally, by rearranging the parameters
as $b_i \to a_i \ (1 \le i \le n ), \ c_k \to b_k \ (1 \le k \le m), \ d \to
c, \  e \to d, \ f \to e$, we have the desired result.
\end{proof}
In the case when $m=1$ and $y_1 = 1$, \eqref{IDT1} reduces to

\begin{eqnarray}\label{m1IDT1}
&&
\sum_{\gamma \in {\Bbb N}^{n}}
q^{|\gamma|} 
\frac{\Delta({x} q^{\gamma})}{\Delta({x})}
\prod_{1 \le i, j \le n}
\frac{(a_j x_i / x_j)_{\gamma_i}} {(q x_i / x_j)_{\gamma_i}}
\prod_{1 \le i \le n}
\frac{(b x_i, c x_i)_{\gamma_i}} 
{(e x_i,  (A b c q^{1-N}/ d e) x_i)_{\gamma_i}}
\\
&&
\quad 
\times
\frac{(q^{-N})_{|\gamma|}}
{(d)_{|\gamma|}}
\ = \
\frac
{(d e  / A c, d e /A b)_N}
{(d e^{2}/ A b c, d)_N}
\prod_{ 1 \le i \le n}
\frac
{((d e a_i / A b c) x_i^{-1} )_N}
{((d e / A b c) x_i^{-1} )_N}
\nonumber\\
&& \quad \quad 
\times \
{}_{2 n +6} W_{2 n + 5}
\left[d e^{2} q^{-1}/ A b c; 
\{ (e / a_i) x_i \}_{n},
\{ (d e q^{-1} / A b c ) x_i^{-1} \}_{n}, 
e/b, e/c, q^{-N};q ; d q^N
\right].
\nonumber
\end{eqnarray}

In the case when $n=1$ and $x_1 = 1$, \eqref{IDT1} reduces to

\begin{eqnarray}\label{n1IDT1}
&&
{}_{m+3} \phi_{m+2}
\left[
\begin{matrix}
q^{-N}, a, \{ b_k y_k \}_{m}, c \\
d, \{ e y_k \}_{m}, a B c q^{1-N} / d e^m
\end{matrix}
; q; q
\right]
\\ 
&&
\quad \quad  
\ = \
\frac{(d e^m /a B, d e^m/ B c)_N}
{(d, d e^m/ a B c)_N}
\prod_{ 1 \le k \le m}
\frac
{((d e^m b_k / a B c) y_k)_N}
{((d e^{m+1}/ a B c) y_k )_N}
\nonumber
\\
&& \quad \quad \quad \quad 
\times
\MW{m, 2}{\{e/ b_k \}_m \\ \{y_k\}_m}{d e^{m+1} q^{-1}/ a B c}
{ e/ c,  e / a }
{q^{-N}, d e^m / A B c}
{q; d q^N}.
\nonumber
\end{eqnarray}

\begin{rem}
The reason why we call \eqref{IDT1} as the inverse of the (non-balanced)
duality transformation \eqref{DT1} is that the transformations which one obtain by applying
\eqref{DT1} and \eqref{IDT1} in both order turn out to be identical as a (multiple) 
hypergeometric series.
Note also that \eqref{IDT1} can be obtained by relabeling parameters in \eqref{DT1}
appropriately. However, we will use special cases of these transformations
to establish an $A_n$ generalization of basic hypergeometric transformation formulas in the next section.
 
In the case when $m=n=1$, \eqref{IDT1} reduces to
\begin{eqnarray}\label{m1n1IDT1}
&&
{}_{4} \phi_{3}
\left[
\begin{matrix}
q^{-N}, a, b, c \\
d, e , a b c q^{1-N} / d e
\end{matrix}
; q; q
\right]
=
\frac
{(d e  / b c, d e /a c, d e/ a b )_N}
{(d e^{2}/ a b c, d e / a b c, d)_N}
\\
&& \quad \quad  \quad \quad 
\times
{}_{8} W_{7}
\left[d e^{2} q^{-1}/ a b c; 
e / a, d e q^{-1} / a b c,
e/b, e/c, q^{-N};q ; d q^N
\right].
\nonumber
\end{eqnarray}
\end{rem}

We also mention that  by letting  $e=c$ in \eqref{m1IDT1}  
and then rearranging the parameter $d$ as $Ab q^{1-N} / c$, we 
recover an $A_n$ generalization of Pfaff-Saalsch\"{u}tz summation 
formula for terminating balanced ${}_3 \phi_2$ series due to Milne
(Theorem 4.15 in \cite{Milne3})

\begin{eqnarray}\label{PSSum1}
&&
\sum_{\gamma \in {\Bbb N}^{n}}
q^{|\gamma|} 
\frac{\Delta({x} q^{\gamma})}{\Delta({x})}
\prod_{1 \le i, j \le n}
\frac{(a_j x_i / x_j)_{\gamma_i}} {(q x_i / x_j)_{\gamma_i}}
\prod_{1 \le i \le n}
\frac{(b x_i)_{\gamma_i}} 
{(c x_i)_{\gamma_i}}
\\
&& \quad \quad \quad \quad 
\times 
\frac{(q^{-N})_{|\gamma|}}
{(A b q^{1-N} /c )_{|\gamma|}}
=
\frac
{(c/ b)_N}
{(c/ A b)_N}
\prod_{ 1 \le i \le n}
\frac
{((c/ a_i) x_i )_N}
{(c x_i )_N}.
\nonumber
\end{eqnarray}

\begin{rem}
In the case when $n=1$, \eqref{PSSum1} reduces to Jackson`s
Pfaff-Saalsch\"{u}tz summation formula for terminating balanced ${}_3 \phi_2$
series (the formula (1.7.2) in \cite{GR1}) 
\begin{equation}
{}_{3} \phi_2 \left[
	\begin{matrix}
		a, b,  q^{-N}\\
 			c, ab q^{1-N} /c
\end{matrix}
;q; q
\right] 
= \frac{(c/a, c/b)_N}{(c, c/ab)_N}.
\end{equation}
\end{rem}

\begin{rem}
Similarly to the expression \eqref{HBDT1} for the balanced duality transformation formula 
\eqref{BDT1}, \eqref{IDT1} can be stated in more general form:
 
\begin{eqnarray}\label{HIDT1}
&&
\sum_{\gamma \in {\Bbb N}^n}
q^{|\gamma|} 
\frac{\Delta (x q^{\gamma})}{\Delta (x)}
\prod_{1 \le i, j \le n}
\frac{(a_j x_i / x_j)_{\gamma_i}} {(q x_i / x_j)_{\gamma_i}}
\prod_{1 \le i \le n, 1 \le k \le m}
\frac{(b_k x_i y_k)_{\gamma_i}}{(c x_i y_k)_{\gamma_i}}
\\
&& 
\quad 
\times
\frac{(q^{-N})_{|\gamma|}} 
{(A B q^{1-N} /c^m)_{|\gamma|}}
=
\
\frac{(q)_N}{( c^m / A B )_N}
\
\sum_{\delta \in {\Bbb N}^m, \ |\delta| = N}
\frac{\Delta (y q^{\delta})}{\Delta (y)}
\nonumber\\
&& 
\quad \quad 
\times
\prod_{1 \le k, l \le m}
\frac{((c / b_l^{}) y_k / y_l)_{\delta_k}}{(q y_k / y_l)_{\delta_k}}
\prod_{1 \le i \le n, 1 \le k \le m}
\frac{((c /a_i^{}) x_i y_k )_{\delta_k}}
{(c x_i y_k)_{\delta_k}}. 
\nonumber
\end{eqnarray}
Note that \eqref{HIDT1} can be obtained by taking the coefficient of $u^N$ in
the multiple basic Euler transformation formula \eqref{ETG}.
We remark that \eqref{IDT1} corresponds to the $m \ge 2$ case of \eqref{HIDT1} by rearrangement of parameters 
and in the case when $m=1$, \eqref{HIDT1} reduces to $A_n$ Pfaff-Saalsch\"{u}tz summation formula \eqref{PSSum1}.
\end{rem}


\subsection{${}_4 \phi_3$ transformation formulas of type $A$}

\medskip

\noindent
{\bf { \large Sears transformation of type $A$}}

\medskip

\begin{prop}
\begin{eqnarray}\label{ST1}
&&
\sum_{\gamma \in {\Bbb N}^n}
q^{|\gamma|} 
\frac{\Delta (x q^{\gamma})}{\Delta (x)}
\prod_{1 \le i, j \le n}
\frac{(b_j x_i / x_j)_{\gamma_i}} {(q x_i / x_j)_{\gamma_i}}
\prod_{1 \le i \le n, 1 \le k \le m}
\frac{(c_k x_i y_k )_{\gamma_i}}{(d x_i y_k)_{\gamma_i}}
\\
&& \quad \quad \times
\frac{(q^{-N}, a )_{|\gamma|}}
{(e, a B C q^{1-N}/ d^m e)_{|\gamma|}}
=
\frac{(e / a, d^m e /B C)_N}
{(e, d^m e / a B C)_N} 
\sum_{\delta \in {\Bbb N}^m}
q^{|\delta|} 
\frac{\Delta (y q^{\delta})}{\Delta (y)}
\nonumber 
\\
&& \quad \quad \quad  \times
\frac{(q^{-N}, a )_{|\delta|}}
{(q^{1-N} a / e, d^m e / B C)_{|\delta|}}
\prod_{1 \le k, l \le m}
\frac{((d / c_l^{}) y_k / y_l)_{\delta_k}}{(q y_k / y_l)_{\delta_k}}
\prod_{1 \le i \le n, 1 \le k \le m}
\frac{((d / b_i^{}) x_i y_k )_{\delta_k}}
{(d x_i y_k )_{\delta_k}}
.
\nonumber
\end{eqnarray}
\end{prop}

\begin{proof}
Replace the parameters in \eqref{BDT1} $d, e$ and $f$ with $aq/ d, aq/e$ and $aq/f$ respectively.
Then put $a=0$. Then rearranging parameters $f \to d$ and $ de f^m q^{N-1} / B C \to a$ leads to 
the desired identity.
\end{proof}

 In the case when $m=1$ and $y_1= 1$, \eqref{ST1} reduces to
\begin{eqnarray}\label{m1ST1}
&&
\sum_{\gamma \in {\Bbb N}^n}
q^{|\gamma|} 
\frac{\Delta (x q^{\gamma})}{\Delta (x)}
\frac{(q^{-N}, a )_{|\gamma|}}
{(e, a B c q^{1-N}/ d e )_{|\gamma|}}
\prod_{1 \le i, j \le n}
\frac{(b_j x_i / x_j)_{\gamma_i}} {(q x_i / x_j)_{\gamma_i}}
\prod_{1 \le i \le n}
\frac{(c x_i )_{\gamma_i}}{(d x_i)_{\gamma_i}}
\\
&& \quad \quad \quad \quad 
= \
\frac{(e / a, d e / B c)_N}
{(e, d e / a B c)_N} \
\
{}_{n+3} \phi_{n+2}
\left[
	\begin{matrix}
		q^{-N}, a, \{ (d / b_i) x_i \}_n, d /c \\
		q^{1-N} a / e, \{d x_i\}_n, 
		d e /B c 
	\end{matrix}; q, q
\right]
. \nonumber
\end{eqnarray}

\begin{rem}
The multiple Sears transformation formula \eqref{ST1} was originally established in \cite{Kaji1}
(the formula (7.1) in \cite{Kaji1}). 
In the case when $m=n=1$ and $x_1 = y_1 = 1$, \eqref{ST1} reduces to the Sears transformation formula 
for terminating balanced ${}_4 \phi_3 $ series:

\begin{eqnarray}\label{SearsT1}
{}_{4} \phi_{3} 
\left[
	\begin{matrix}
		q^{-N}, a, b, c \\
		d, e, a b c q^{1-N} / d e
	\end{matrix};
q, q
\right]
&=& 
\frac{(e/a, d e / b c)_N}
{(e, d e / a b c)_N}
{}_4 \phi_3
\left[
	\begin{matrix}
		q^{-N}, a, d/b, d/c \\
		d, a q^{1-N}/e, d e / b c 
	\end{matrix};
q, q
\right].
\end{eqnarray}
 Further informations for multiple Sears transformation \eqref{ST1} can 
also be found in \cite{KajiS}.
\end{rem}

We shall add a few remarks that we have missed in our previous 
works in \cite{Kaji1} and \cite{KajiS}.

Let $N$ tend to infinity in \eqref{ST1}, we get the following 
transformation formula for multiple nonterminating ${}_3 \phi_2$
series of type $A$:

\begin{cor}
\begin{eqnarray}\label{NT32}
&&
\sum_{\gamma \in {\Bbb N}^n}
\left(  \frac{d^m e}{a B C} \right)^{|\gamma|} 
\frac{\Delta (x q^{\gamma})}{\Delta (x)}
\prod_{1 \le i, j \le n}
\frac{(b_j x_i / x_j)_{\gamma_i}} {(q x_i / x_j)_{\gamma_i}}
\prod_{1 \le i \le n, 1 \le k \le m}
\frac{(c_k x_i y_k )_{\gamma_i}}{(d x_i y_k)_{\gamma_i}}
\\
&&
\quad 
\times \
\frac{( a )_{|\gamma|}}
{(e)_{|\gamma|}}
\ = \
\frac{(e / a, d^m e /B C)_\infty}
{(e, d^m e / a B C)_\infty} 
\sum_{\delta \in {\Bbb N}^m}
\left( \frac{e}{a} \right)^{|\delta|} 
\frac{\Delta (y q^{\delta})}{\Delta (y)}
\nonumber\\
&& 
\quad \quad 
\times
\frac{( a )_{|\delta|}}
{(d^m e / B C)_{|\delta|}}
\prod_{1 \le k, l \le m}
\frac{((d / c_l^{}) y_k / y_l)_{\delta_k}}{(q y_k / y_l)_{\delta_k}}
\prod_{1 \le i \le n, 1 \le k \le m}
\frac{((d / b_i^{}) x_i y_k)_{\delta_k}}
{(d x_i y_k)_{\delta_k}}
\nonumber
\end{eqnarray}
holds if it satisfy the convergence condition
$\max (|d^m e / a BC|, |e/a| ) < 1$.
\end{cor}

\begin{rem}
\eqref{NT32} was already established as the equation (3.2) in 
\cite{KajiS}. In the case when $m=n=1$ and $x_1 = y_1= 1$, \eqref{NT32} reduces to
the following nonterminating ${}_3 \phi_2$ transformation formula
\begin{eqnarray}\label{1DNT32}
{}_{3} \phi_{2} 
\left[
	\begin{matrix}
		 a, b, c \\
		d, e
	\end{matrix};
q, \frac{d e}{a b c}
\right]
&=& 
\frac{(e/a, d e / b c)_\infty}
{(e, d e / a b c)_\infty}
{}_3 \phi_2
\left[
	\begin{matrix}
		 a, d/b, d/c \\
		d,  d e / b c 
	\end{matrix};
q, \frac{e}{a}
\right].
\end{eqnarray}
\end{rem}

In the case when $m=1$ and $y_1 = 1$, \eqref{NT32} reduces to 

\begin{eqnarray}\label{m1-NT32}
&&
\sum_{\gamma \in {\Bbb N}^n}
\left(  \frac{d e}{a B c} \right)^{|\gamma|} 
\frac{\Delta (x q^{\gamma})}{\Delta (x)}
\frac{( a )_{|\gamma|}}
{(e)_{|\gamma|}}
\prod_{1 \le i, j \le n}
\frac{(b_j x_i / x_j)_{\gamma_i}} {(q x_i / x_j)_{\gamma_i}}
\prod_{1 \le i \le n}
\frac{(c x_i )_{\gamma_i}}{(d x_i)_{\gamma_i}}
\\
&&
\quad 
=
\frac{(e / a, d e /B c)_\infty}
{(e, d e / a B c)_\infty} 
{}_{n+2} \phi_{n+1}
\left[
\begin{matrix}
d/c,  a, \{ (d / b_i) x_i \}_n   \\
d e / B c,  \{ d x_i \}_n
\end{matrix}
;
q;
\frac{e}{a}
\right]
\nonumber
\end{eqnarray}
holds if it satisfy the convergence condition
$\max (|d^m e / a BC|, |e/a| ) < 1$.

Here we give a remark on the convergence of  multiple series.

\begin{rem}{ \bf (Convergence of the multiple series)}
In the course of obtaining an infinite multiple sum identity such as \eqref{NT32}, 
one needs to ensure the limiting procedure. To justify the process, one is 
required to use the dominated convergence theorem. 
Furthermore, to find the convergence condition of the dominated series such as the series 
in \eqref{NT32}, one can quote the ratio test for multiple power series
in classical work by J.Horn \cite{Horn}. 
Since the details of such justifications in this paper are all in the same line as the corresponding 
discussions in Milne-Newcomb \cite{MilNew1} (see also Milne \cite{Milne3} and Milne-Newcomb \cite{MilNew2}),
 we shall omit in this paper.    
\end{rem}

We shall propose the following special case, namely a transformation formula 
between  $A_n$ ${}_{m+2} \phi_{m+1}$ series 
with integer parameter differences and $A_m$ terminating balanced ${}_{n+2} \phi_{n+1}$ series:

\begin{eqnarray}\label{MK-PS}
&&
\sum_{\gamma \in {\Bbb N}^n}
\left(  q^{1- |M |}  /A\right)^{|\gamma|} 
\frac{\Delta (x q^{\gamma})}{\Delta (x)}
\prod_{1 \le i, j \le n}
\frac{(a_j x_i / x_j)_{\gamma_i}} {(q x_i / x_j)_{\gamma_i}}
\\
&&
\quad 
\times
\frac{( b )_{|\gamma|}}
{( b q )_{|\gamma|}}
\prod_{1 \le i \le n, 1 \le k \le m}
\frac{( c q^{m_k} x_i y_k )_{\gamma_i}}{(c x_i y_k)_{\gamma_i}}
\nonumber\\
&& 
\quad \quad 
=
\frac{( q,  q^{1-|M|} b/A )_\infty}
{( b,  q^{1-|M|} /A )_\infty}
\sum_{\delta \in {\Bbb N}^m}
q^{|\delta|} 
\frac{\Delta (y q^{\delta})}{\Delta (y)}
\prod_{1 \le k, l \le m}
\frac{( q^{- m_l} y_k / y_l)_{\delta_k}}{(q y_k / y_l)_{\delta_k}}
\nonumber\\
&& \quad \quad \quad 
\times
\frac{( b)_{|\delta|}}
{(  q^{1-|M|}  b / A)_{|\delta|}}
\prod_{1 \le i \le n, 1 \le k \le m}
\frac{((c / a_i^{}) x_i y_k)_{\delta_k}}
{(c x_i y_k)_{\delta_k}}.
\nonumber
\end{eqnarray}

In the case when $n=1$ and $x_1 = 1$,  the multiple series in the right hand side in \eqref{MK-PS} can be summed by
the following $A_n$ generalization of $q$-Pfaff-Saalsch\"{u}tz summation formula
\begin{eqnarray}\label{R-PSSum1}
&&
\sum_{\gamma \in {\Bbb N}^{n}}
q^{|\gamma|} 
\frac{\Delta({x} q^{\gamma})}{\Delta({x})}
\prod_{1 \le i, j \le n}
\frac{(q^{- m_j} x_i / x_j)_{\gamma_i}} {(q x_i / x_j)_{\gamma_i}}
\prod_{1 \le i \le n}
\frac{(b x_i)_{\gamma_i}} 
{(c x_i)_{\gamma_i}}
\\
&&  \quad \quad \quad  \times \
\frac{(a)_{|\gamma|}}
{(a b q^{1-|M|} /c )_{|\gamma|}}
\ = \
\frac
{(c/ b)_{|M|}}
{(c/ a b)_{|M|}}
\prod_{ 1 \le i \le n}
\frac
{((c/ a) x_i )_{m_i}}
{(c x_i )_{m_i}},
\nonumber
\end{eqnarray}
which can be obtained from \eqref{PSSum1} by elementary polynomial argument (see the proof of Corollary 4.1),
to recover Gasper's $q$-analogue of Minton-Karlsson summation formula for ${}_{n+2} \phi_{n+1}$
series \cite{Gas1}:

\begin{eqnarray}\label{MK-SumB1}
&&
{}_{n+2} \phi_{n+1}
\left[
\begin{matrix}
a,  b, \{ c q^{m_i} x_i \}_n   \\
b q,  \{ c x_i \}_n
\end{matrix}
;
q;
a^{-1} q^{1-|M|}
\right]
\ = \
b^{|M|} 
\
\frac
{(q, b q / a)_\infty}
{(b q, q/a )_\infty}
\prod_{1 \le i \le n}
\frac{((c/b ) x_i)_{m_i}}
{(c x_i)_{m_i}}. 
\end{eqnarray}

In the case when $m=1$ and $y_1 = 1$,  \eqref{MK-PS} reduces to:

\begin{eqnarray}\label{m1-MK-PS}
&&
\sum_{\gamma \in {\Bbb N}^n}
\left(  q^{1- m} /A\right)^{|\gamma|} 
\frac{\Delta (x q^{\gamma})}{\Delta (x)}
\prod_{1 \le i, j \le n}
\frac{(a_j x_i / x_j)_{\gamma_i}} {(q x_i / x_j)_{\gamma_i}}
\prod_{1 \le i \le n}
\frac{( c q^{m} x_i )_{\gamma_i}}{(c x_i)_{\gamma_i}}
\\
&& \quad \quad \quad \times
\frac{( b )_{|\gamma|}}
{( b q )_{|\gamma|}}
\ = \
\frac{( q,  q^{1-M} b / A )_\infty}
{( b,  q^{1-M} / A )_\infty}
{}_{n+2} \phi_{n+1}
\left[
\begin{matrix}
q^{-M},  b, \{ (c/ a_i) x_i \}_n   \\
q^{1-M} b /A,  \{ c x_i \}_n
\end{matrix}
;
q;
q
\right].
\nonumber
\end{eqnarray}

In the case when $m=n=1$ and $x_1 = y_1 = 1$,  \eqref{MK-PS} reduces to:

\begin{eqnarray}\label{m1n1-MK-PS}
{}_{3} \phi_{2}
\left[
\begin{matrix}
a,  b,  c q^M   \\
b q, c
\end{matrix}
;
q;
q^{1-M} / a
\right]
&=&
b^{M} 
\
\frac
{(q, b q / a, c / b)_\infty}
{(b q, q / a, c )_\infty}.
\end{eqnarray}

\medskip

It may be of worth to note that,  furthermore, if we replace $e \to a B C u / d^m$ and take 
limit $a \to \infty$ in \eqref{NT32}, we recover  multiple basic Euler transformation formula of type $A$
\eqref{ETG}

\begin{eqnarray*}
&&
\sum_{\gamma \in {\Bbb N}^n}
u^{|\gamma|} 
\frac{\Delta (x q^{\gamma})}{\Delta (x)}
\prod_{1 \le i, j \le n}
\frac{(a_j x_i / x_j)_{\gamma_i}} {(q x_i / x_j)_{\gamma_i}}
\prod_{1 \le i \le n, 1 \le k \le m}
\frac{(b_k x_i y_k)_{\gamma_i}}{(c x_i y_k)_{\gamma_i}}
\\
&&
\quad =
\frac{(A Bu/ c^m)_\infty}
{(u)_\infty}
\sum_{\delta \in {\Bbb N}^m}
 \left( \frac{A B u}{c^m} \right)^{|\delta|} 
\frac{\Delta (y q^{\delta})}{\Delta (y)}
\nonumber\\
&& 
\quad \quad 
\times
\prod_{1 \le k, l \le m}
\frac{((c / b_l^{}) y_k / y_l)_{\delta_k}}{(q y_k / y_l)_{\delta_k}}
\prod_{1 \le i \le n, 1 \le k \le m}
\frac{((c /a_i^{}) x_i y_k )_{\delta_k}}
{(c x_i y_k )_{\delta_k}},
\nonumber
\end{eqnarray*}
when $ max(|u|, \displaystyle{ \left| AB u / c^m \right|} ) <1 $.
\vspace{12mm}

\noindent
{\bf { \large Reversing version}}

\medskip

\begin{prop}

Under the balancing condition
\begin{equation}
a^{m} B c q^{1-N} = d E f,
\end{equation}
we have the following:

\begin{eqnarray}\label{ST2}
&&
\sum_{\gamma \in {\Bbb N}^{n}}
q^{|\gamma|} 
\frac{\Delta({x} q^{\gamma})}{\Delta({x})}
\prod_{1 \le i, j \le n}
\frac{(b_j x_i / x_j)_{\gamma_i}} {(q x_i / x_j)_{\gamma_i}}
\prod_{1 \le i\le n}
\frac{(c x_i)_{\gamma_i}} 
{(d  x_i)_{\gamma_i}}
\prod_{1 \le k \le m}
\frac{( a y_k)_{|\gamma|}}
{( e_k y_k)_{|\gamma|}}
\\
&&
\quad 
\times
\frac{(q^{-N})_{|\gamma|}}
{(f)_{|\gamma|}}
\ = \
\frac
{( E f /a^m B)_N}
{(f )_N}
\prod_{ 1 \le k \le m}
\frac
{( a y_k )_N}
{(e_k  y_k )_N}
\prod_{ 1 \le i \le n}
\frac
{((  E f / a^m  c ) z_i )_N}
{(( E f / a^m B c)  z_i )_N}
\nonumber\\
&& 
\quad \quad 
\times
\sum_{\delta \in {\Bbb N}^{m} }
q^{|\delta|}
\frac{\Delta ({w} q^{\delta})}{\Delta ({w})}
\prod_{1 \le k, l \le m}
\frac{((e_l/a) w_k / w_l)_{\delta_k}}{(q w_k / w_l)_{\delta_k}}
\prod_{1 \le k \le m}
\frac{(( f/ a) w_k )_{\delta_k}}
{(( d E f / a^{m+1} B c) w_k )_{\delta_k}}
\nonumber\\
&&
\quad \quad \quad 
\times
\frac{(q^{-N})_{|\delta|}}{(E f /  a^m  B )_{|\delta|}}
\prod_{1 \le i \le n}
\frac{( ( E f /a^m b_i c) z_i)_{|\delta|}}
{( ( E f /a^m c) z_i)_{|\delta|}}
\nonumber
\end{eqnarray}
where $ z_i  = b_i / B x_i, \ ( 1 \le i \le n)$ and 
$ w_k   =  y_k^{-1} , \ ( 1 \le k \le m)$ respectively.
\end{prop}

\begin{proof}
Rewrite the parameters $c_k \to aq/ c_k \ ( 1 \le k \le m)$ and  $e \to aq/ e$ 
in \eqref{BDT1}. Then  put $a=0$. Finally by rearranging the parameters
$ f \to a, \ d \to c, \ {B d f^m q^{1-N}}/ {C e} \to d,  
\ c_k \to e_k \quad (1 \le k \le m),  \ e \to f $, we arrive at the desired identity.
\end{proof}

 In the case when $m=1$ and $y_1=1$, \eqref{ST2} reduces to

\begin{eqnarray}\label{m1ST2}
&&
\sum_{\gamma \in {\Bbb N}^{n}}
q^{|\gamma|} 
\frac{\Delta({x} q^{\gamma})}{\Delta({x})}
\prod_{1 \le i, j \le n}
\frac{(b_j x_i / x_j)_{\gamma_i}} {(q x_i / x_j)_{\gamma_i}}
\prod_{1 \le i\le n}
\frac{(c x_i)_{\gamma_i}} 
{( d x_i)_{\gamma_i}}
\\
&& 
\quad \quad
\times
\frac{(q^{-N}, a)_{|\gamma|}}
{(e, f)_{|\gamma|}}
\ = \
\frac
{(e f / a B, a)_N}
{(e, f)_N}
\prod_{ 1 \le i \le n}
\frac
{(( e f / a c )z_i )_N}
{((e f / a B c) z_i )_N}
\nonumber\\
&& 
\quad \quad \quad \quad \quad 
\times
{}_{n+3} \phi_{n+2}
\left[
\begin{matrix}
q^{-N}, e/ a, f / a, \{( e f / a b_i c) z_i\}_{n}
\\
 def/ a^2 B c, e f / a B, \{( e f / a c) z_i\}_{n}
\end{matrix}
; q; q
\right],
\nonumber
\end{eqnarray}
when the following balancing condition holds:
\begin{equation}
a  B c q^{1-N} = d e f, 
\end{equation}
and where $z_i = b_i / B x_i, i= 1, 2, \cdots n$.
By letting $f=a$ in \eqref{m1ST2} and relabeling the parameters 
appropriately, we also recover the $A_n$ Pfaff-Saalsch\"{u}tz summation 
\eqref{PSSum1}.

\begin{rem}
\eqref{ST2} has already appeared as the 2nd Sears transformation formula 
(Proposition 7.2.) in \cite{Kaji1}, up to relabeling parameters.
 In the case when $m=n=1$ and $x_1 = y_1 =1$, \eqref{ST2} reduces to 
a transformation formula for terminating balanced ${}_ 4 \phi_3 $ series
\begin{eqnarray}\label{m1n1ST2}
&&
{}_{4} \phi_{3}
\left[
\begin{matrix}
q^{-N}, a, b, c 
\\
d, e, f
\end{matrix}
; q; q
\right]
\ = \ 
\frac
{(e f / a b, e f/ a c, a)_N}
{(e, f, e f / a b c)_N}
\\
&&
\quad \quad \quad \quad \quad \quad \quad \quad \quad \quad
\times
{}_{4} \phi_{3}
\left[
\begin{matrix}
q^{-N},e f / a b c, e/a, f/a
\\
def/ a^2 b c, e f / a b, e f / a c
\end{matrix}
; q; q
\right]
\nonumber
\end{eqnarray}
when
\begin{equation}
a  b c q^{1-N} = d e f.  
\end{equation}
Note that this is obtained by reversing the order of the summation of 
the Sears transformation \eqref{SearsT1} and is also verified by iterating 
Sears transformation twice in a proper fashion. 
\end{rem}

Let $N \to \infty$ in \eqref{ST2}. Then we have

\begin{eqnarray}\label{Hall1}
&&
\sum_{\gamma \in {\Bbb N}^{n}} 
x_1^{- \gamma_1} \cdots  x_n^{- \gamma_n}
\left( \frac{E f}{a^m B c} \right)^{|\gamma|} 
q^{ e_2 ( \gamma )} \
\frac{\Delta({x} q^{\gamma})}{\Delta({x})}
\prod_{1 \le i, j \le n}
\frac{(b_j x_i / x_j)_{\gamma_i}} {(q x_i / x_j)_{\gamma_i}}
\prod_{1 \le i\le n}
{(c x_i)_{\gamma_i}} 
\\
&& 
\quad 
\times
\left[{(f)_{|\gamma|}}\right]^{-1}
\prod_{1 \le k \le m}
\frac{( a y_k)_{|\gamma|}}
{( e_k y_k)_{|\gamma|}}
\nonumber\\
&&
\quad \quad 
=
\frac
{( E f /a^m B)_\infty}
{(f )_\infty}
\prod_{ 1 \le k \le m}
\frac
{( a y_k )_\infty}
{(e_k  y_k )_\infty}
\prod_{ 1 \le i \le n}
\frac
{((  E f / a^m  c ) z_i )_\infty}
{(( E f / a^m B c)  z_i )_\infty}
\nonumber\\
&&
\quad \quad \quad 
\times
\sum_{\delta \in {\Bbb N}^{m} }
w_1^{-\delta_1} \cdots w_m^{-\delta_m} a^{|\delta|}
q^{ e_2 ( \delta )} \
\frac{\Delta ({w} q^{\delta})}{\Delta ({w})}
\prod_{1 \le k, l \le m}
\frac{((e_l/a) w_k / w_l)_{\delta_k}}{(q w_k / w_l)_{\delta_k}}
\prod_{1 \le k \le m}
{(( f/ a) w_k )_{\delta_k}}
\nonumber\\
&& \quad \quad \quad \quad 
\times
\left[{(E f /  a^m  B )_{|\delta|}}\right]^{-1}
\prod_{1 \le i \le n}
\frac{( ( E f /a^m b_i c) z_i)_{|\delta|}}
{( ( E f /a^m c) z_i)_{|\delta|}}, 
\nonumber
\end{eqnarray}
where $ z_i  = b_i / B x_i, \ ( 1 \le i \le n)$, 
$ w_k   =  y_k^{-1} , \ ( 1 \le k \le m)$ respectively 
and $e_2( \gamma )$ is the second elementary symmetric function 
of the variable $\gamma = ( \gamma_1, \gamma_2, \cdots , \gamma_n )$.

The case $m=1$ and $y_1 = 1$ of \eqref{Hall1} is
\begin{eqnarray}\label{m1Hall1}
&&
\sum_{\gamma \in {\Bbb N}^{n}}
x_1^{- \gamma_1} \cdots  x_n^{- \gamma_n}
\left( \frac{e f}{a B c}\right)^{|\gamma|} 
q^{ e_2 ( \gamma )} \
\frac{\Delta({x} q^{\gamma})}{\Delta({x})}
\prod_{1 \le i, j \le n}
\frac{(b_j x_i / x_j)_{\gamma_i}} {(q x_i / x_j)_{\gamma_i}}
\\
&& 
\quad \quad 
\times 
\frac{( a)_{|\gamma|}}
{(e, f)_{|\gamma|}}
\prod_{1 \le i\le n}
{(c x_i )_{\gamma_i}}
\ = \
\frac
{( e f /a B, a)_\infty}
{(e, f)_\infty}
\prod_{ 1 \le i \le n}
\frac
{((e f / a c)z_i )_\infty}
{((e f /a B c) z_i)_\infty}
\nonumber \\
&& 
\quad  \quad \quad \quad \quad \quad
\times 
{}_{n+2} \phi_{n+1}
\left[
\begin{matrix}
 e/ a, f/ a, \{( e f /a B c) z_i \}_{n}
\\
 e f / a B, \{( e f/ a c) z_i \}_{n}
\end{matrix}
; q; a
\right], 
\nonumber 
\end{eqnarray}
where $z_i = b_i / B x_i$ for $ i = 1, \cdots, n$.
We mention that one yield an $A_n$ Gauss summation theorem by putting $f=a$ and relabeling the parameters: 
\begin{eqnarray}\label{2ndGaussSum}
&&
\sum_{\gamma \in {\Bbb N}^{n}}
x_1^{- \gamma_1} \cdots  x_n^{- \gamma_n}
\left( \frac{e f}{a B c}\right)^{|\gamma|} 
q^{ e_2 ( \gamma )} \
\frac{\Delta({x} q^{\gamma})}{\Delta({x})}
\prod_{1 \le i, j \le n}
\frac{(b_j x_i / x_j)_{\gamma_i}} {(q x_i / x_j)_{\gamma_i}}
\\
&& 
\quad \quad 
\times 
{(c)_{|\gamma|}}^{-1}
\prod_{1 \le i\le n}
{(a x_i )_{\gamma_i}}
\ = \
\frac
{( c / B)_\infty}
{(c)_\infty}
\prod_{ 1 \le i \le n}
\frac
{((c B / a b_i ) x_i^{-1} )_\infty}
{((c /a b_i) x_i^{-1})_\infty}.
\nonumber 
\end{eqnarray}

\begin{rem}
In the case when  $m=n=1$ and $x_1 = y_1 = 1$, \eqref{Hall1} reduces to ${}_3 \phi_2$ transformation 
formula
\begin{eqnarray}\label{m1n1Hall1}
&&
{}_{3} \phi_{2}
\left[
\begin{matrix}
a, b, c
\\
 e, f 
\end{matrix}
; q; \frac{e f}{a b c}
\right]
\ = \
\frac
{(e f / a b, e f/ a c, a )_\infty}
{(e, f, e f / a b c)_\infty}
{}_{3} \phi_{2}
\left[
\begin{matrix}
e/a, f/a, e f / a b c  
\\
e f / a b, e f / ac 
\end{matrix}
; q; a
\right].
\nonumber
\end{eqnarray}
In the case when  $mn=1$ and $x_1 = 1$, \eqref{2ndGaussSum} reduces to the basic Gauss summation formation formula 
for ${}_2 \phi_1$ 
\begin{eqnarray}\label{n1GaussSum}
&&
{}_{2} \phi_{1}
\left[
\begin{matrix}
a, b
\\
 c, 
\end{matrix}
; q; \frac{c}{a b}
\right]
\ = \
\frac
{(c/ a, c/ b)_\infty}
{(c, c / a b )_\infty}
.
\nonumber
\end{eqnarray}
We also remark that \eqref{Hall1} and \eqref{2ndGaussSum} have already appeared in our previous work 
\cite{KajiR}. But they contain errors, so we restate them here.
\end{rem}


\subsection{ ${}_8 W_7$ transformations}

\medskip

\noindent
{\bf { \large Nonterminating ${}_8 W_7 $ transformation formula}}

\medskip

By taking the limit $N \to \infty$ in \eqref{BDT1}, we have the following:
\begin{prop}
Assume that ${\displaystyle \left| \frac{a^{m+1} q^{m+1} }{B C d e f^{m}} x_i^{-1} \right| <1 }$
for all $i = 1, \cdots , n$ and $|f y_k^{-1}| < 1$ for all $k = 1, \cdots , m$. 
Then we have
\begin{eqnarray}\label{NT87}
&&
\sum_{\gamma \in {\Bbb N}^{n}}
x_1^{-\gamma_1} \cdots x_n^{-\gamma_n} 
\left(
\frac{a^{m+1} q^{m+1} }{B C d e f^{m}} 
\right)^{|\gamma|} 
q^{e_2 ( \gamma )} \ 
\frac{\Delta({x}
q^{\gamma})}
{\Delta({x})}
\prod_{1 \le i \le n}
\frac{1-  a q^{|\gamma| + \gamma_i} x_i}
{1-  a x_i}\\
&&
\quad 
\times
\prod_{1 \le i, j \le n}
\frac{(b_j x_i / x_j)_{\gamma_i}} {(q x_i / x_j)_{\gamma_i}}
\prod_{1 \le i \le n, 1 \le k \le m}
\frac{(c_k x_i y_k )_{\gamma_i}}
{(( a q / f ) x_i y_k)_{\gamma_i}}
\nonumber
\\
&& 
\quad \quad 
\times
\prod_{1 \le i \le n}
\frac{(a x_i )_{|\gamma|}}
{( ( a q /b_i ) x_i)_{|\gamma|}}
\prod_{1 \le k \le m}
\frac{( f y_k^{-1})_{|\gamma|}}
{(( a q / c_k )  y_k^{-1})_{|\gamma|}}
\nonumber\\
&&
\quad \quad \quad 
\times
\frac{1}{( a q / d, a q / e )_{|\gamma|}}
\prod_{1 \le i \le n}
\left({(d x_i, e x_i )_{\gamma_i}}
\right)
\nonumber
\\
&=&
\frac
{(\mu d f / a, \mu e f / a)_\infty}
{(a q / d, a q / e)_\infty}
\prod_{ 1 \le k \le m}
\frac
{((\mu c_k f / a) y_k, f y_k^{-1} )_\infty}
{(\mu q y_k, (a q /c_k) y_k^{-1} )_\infty}
\nonumber \\
&& 
\quad 
\times
\prod_{ 1 \le i \le n}
\frac
{(a q x_i,  ( \mu b_i f / a )x_i^{-1} )_\infty}
{((a q / b_i) x_i, (\mu f / a) x_i^{-1} )_\infty}
\nonumber\\
&& 
\quad \quad
\times 
\sum_{\delta \in {\Bbb N}^{m} }
y_1^{-\delta_1} \cdots y_m^{-\delta_m} 
f^{|\delta|}  
q^{e_2 ( \delta )} \
\frac{\Delta ({y} q^{\delta})}{\Delta ({y})}
\prod_{1 \le k \le m}
\frac{1 -  \mu q^{|\delta| + \delta_k} y_k }
{1 -  \mu y_k }
\nonumber\\
&&
\quad \quad \quad 
\times
\prod_{1 \le k, l \le m}
\frac{(( a q / c_l^{} f ) y_k / y_l)_{\delta_k}}{(q y_k / y_l)_{\delta_k}}
\prod_{1 \le i \le n, 1 \le k \le m}
\frac{((a q /b_i^{} f) x_i y_k)_{\delta_k}}
{(( a q / f) x_i y_k)_{\delta_k}}
\nonumber
\\
&&
\quad \quad \quad \quad 
\times
\prod_{1 \le k \le m}
\frac{( \mu  y_k)_{|\delta|}}
{(( \mu c_k f / a ) y_k)_{|\delta|}}
\prod_{1 \le i \le n}
\frac{( (\mu f / a )  x_i^{-1})_{|\delta|}}
{( (\mu b_i f/ a ) x_i^{-1})_{|\delta|}}
\nonumber \\
&&
\quad \quad \quad \quad \quad 
\times
\frac{1}{( \mu d f / a,  \mu e f / a )_{|\delta|}}
\prod_{1 \le k \le m}
{((a q / d f ) y_k, (a q / e f ) y_k)_{\delta_k}},
\nonumber
\end{eqnarray}
where
$\mu = a^{m+2} q^{m+1} / B C d e f^{m+1} $.

\end{prop}
\medskip
 
For the justification of this limiting procedure for \eqref{NT87}, see Remark 3.7.

\medskip

In the case when $m=1$ and $y_1 = 1$, \eqref{NT87} reduces to

\begin{eqnarray}\label{m1NT87}
&&
\sum_{\gamma \in {\Bbb N}^{n}}
x_1^{-\gamma_1} \cdots x_n^{-\gamma_n}
\left(
\frac{a^{2} q^{2} }{B c d e f^{}}
\right)^{|\gamma|}
q^{ e_2 ( \gamma )} \ 
\frac{\Delta({x}
q^{\gamma})}
{\Delta({x})}
\prod_{1 \le i \le n}
\frac{1-  a q^{|\gamma| + \gamma_i} x_i}
{1-  a x_i}
\\
&& 
\quad 
\times
\prod_{1 \le i, j \le n}
\frac{(b_j x_i / x_j)_{\gamma_i}} {(q x_i / x_j)_{\gamma_i}}
\prod_{1 \le i \le n}
\frac{(c x_i, d x_i, e x_i)_{\gamma_i}}
{(( a q / f ) x_i)_{\gamma_i}}
\nonumber\\
&&
\quad \quad 
\times
\frac{( f)_{|\gamma|}}
{( a q / c, a q / d, a q / e)_{|\gamma|}}
\prod_{1 \le i \le n}
\frac{(a x_i)_{|\gamma|}}
{( ( a q /b_i ) x_i)_{|\gamma|}}
\nonumber\\
&& 
\quad 
=
\frac
{(\mu c f / a, \mu d f / a, \mu e f / a, f)_\infty}
{(a q / c, a q / d, a q / e, \mu q)_\infty}
\prod_{ 1 \le i \le n}
\frac
{(a q x_i,  ( \mu b_i f / a ) x_i^{-1} )_\infty}
{((a q / b_i) x_i, (\mu f / a) x_i^{-1} )_\infty}
\nonumber\\
&& 
\quad \quad  
\times
{}_{2 n + 4} W_{2 n + 3}
\left[
\mu; \{(a q /b_i^{} f ) x_i \}_{n}
a q / c f,a q / d f,a q / e f,
\{(\mu f / a ) x_i^{-1} \}_{n}
; q; f
\right],\nonumber
\end{eqnarray}
where $\mu = a^{3} q^{2} / B c d e f^{2} $.

\begin{rem}
In the case when $m=n=1$ and $x_1 = y_1 =1$, \eqref{NT87} reduces to 
the following transformation formula for nonterminating ${}_8 W_7$ series
\begin{eqnarray}\label{mn1NT87}
&&
{}_{8} W_{7} 
\left[
	\begin{matrix}
		a; b, c, d, e, f
	\end{matrix};
q; \frac{a^2 q^2}{b c d e f}
\right]
=
\frac
{(\mu b f/ a, \mu c f / a, \mu d f / a, \mu e f / a, a q, f)_\infty }
{(a q / b, a q /c, a q/d, a q/e, \mu q, \mu f / a)_\infty}
\\
&&
\quad \quad \quad \quad \quad 
\times
{}_{8} W_{7} 
\left[
	\begin{matrix}
		\mu ; a q / b f,  a q / c f,  a q / d f,  a q / e f,  
        \mu f / a
	\end{matrix};
q;  f
\right],\nonumber
\end{eqnarray}
where $ \mu = a^3 q^2 /bcde f^2$, 
under the convergence condition $\max(|{a^2 q^2}/{b c d e f}|,| f|) < 1$.
\end{rem}

We mention that by assuming $ e= aq/ f$ in \eqref{m1NT87}, we get the following:

\begin{cor}
Assume that ${ \displaystyle \left|\frac{a  q }{B c d } x_i^{-1} \right| < 1}$ 
for all $i=1, \cdots , n$. Then 
\begin{eqnarray}\label{NT65Sum}
&&
\sum_{\gamma \in {\Bbb N}^{n}}
x_1^{-\gamma_1} \cdots x_n^{-\gamma_n}
\left(
\frac{a  q }{B c d }
\right)^{|\gamma|}
q^{e_2 ( \gamma )} \
\frac{\Delta({x}
q^{\gamma})}
{\Delta({x})}
\prod_{1 \le i \le n}
\frac{1-  a q^{|\gamma| + \gamma_i} x_i}
{1-  a x_i}
\\
&& 
\quad 
\times
\prod_{1 \le i, j \le n}
\frac{(b_j x_i / x_j)_{\gamma_i}} {(q x_i / x_j)_{\gamma_i}}
\prod_{1 \le i \le n}
{(c x_i, d x_i)_{\gamma_i}}
\nonumber\\
&&
\quad \quad 
\times
\left[{( a q / c, a q / d)_{|\gamma|}}\right]^{-1}
\prod_{1 \le i \le n}
\frac{(a x_i)_{|\gamma|}}
{( ( a q /b_i ) x_i)_{|\gamma|}}
\nonumber\\
&=&
\frac
{(a q / B c, a q / B d)_\infty}
{(a q / c, a q / d)_\infty}
\prod_{ 1 \le i \le n}
\frac
{(a q x_i,  (  a q  b_i / B c d ) x_i^{-1} )_\infty}
{((a q / b_i) x_i, (a q / B c d) x_i^{-1} )_\infty}.
\nonumber
\end{eqnarray}
\end{cor}

\begin{rem}
In the case when $n=1$ and $x_1 =1$, \eqref{NT65Sum} reduces to the Bailey summation formula for 
nonterminating ${}_6 W_5$ series (See \cite{GR1}) 
\begin{eqnarray}\label{BaileySum1}
{}_{6} W_{5} 
\left[
	\begin{matrix}
		a; b, c, d
	\end{matrix};
q; \frac{a q}{b c d }
\right]
&=&
\frac
{( a q, a q / c d, a q/ b d, a q / b c )_\infty }
{(a q / b c d, a q / b, a q / c, a q/ d)_\infty}.
\end{eqnarray}
$A_n$ nonterminating ${}_6 W_5$ summation formula \eqref{NT65Sum} is due to S.C.~Milne and
has appeared as Theorem 4.27 in \cite{Milne2} with a different expression (See also Theorem A.4 in 
Milne-Newcomb \cite{MilNew2}).  
Note that \eqref{NT65Sum} can also be obtained by taking the limit $N \to \infty$ 
in $A_n$ Jackson summation formula \eqref{JS1}.
\end{rem}

\medskip

\noindent
{\bf Terminating ${}_8 W_7$ transformation}

\medskip

\begin{prop}
\begin{eqnarray}\label{T87T1}
&&
\sum_{\gamma \in {\Bbb N}^{n}}
x_1^{\gamma_1} \cdots x_n^{\gamma_n}
\left( \frac{a^{m+1} q^{m+N + 1}}{B C d^m e}\right)^{|\gamma|}
q^{ - e_2 ( \gamma ) } \
\frac{\Delta({x} q^{\gamma})}{\Delta({x})}
\prod_{1 \le i \le n}
\frac{1 - a q^{ |\gamma|+ \gamma_i} x_i}
{1 - a x_i}
\\
&& 
\quad 
\times
\prod_{1 \le i, j \le n}
\frac{(b_j x_i / x_j)_{\gamma_i}} {(q x_i / x_j)_{\gamma_i}}
\prod_{1 \le i \le n, 1 \le k \le m}
\frac{(c_k x_i y_k)_{\gamma_i}}
{(( a q / d ) x_i y_k )_{\gamma_i}}
\nonumber\\
&&
\quad \quad 
 \times
\prod_{1 \le i\le n}
\frac{1} {( a q^{N+1} x_i, (a q/ e ) x_i)_{\gamma_i}}
\nonumber\\
&&
\quad \quad \quad 
\times
{(q^{-N}, e)_{|\gamma|}}
\prod_{1 \le i \le n}
\frac{(a x_i)_{|\gamma|}}
{( ( a q /b_i ) x_i)_{|\gamma|}}
\prod_{1 \le k \le m}
\frac{( d y_k^{-1})_{|\gamma|}}
{(( a q / c_k ) y_k^{-1})_{|\gamma|}}
\nonumber\\
&=&
\prod_{ 1 \le k \le m}
\frac{((a q /c_k e) y_k^{-1}, d y_k^{-1} )_N}
{((a q /c_k) y_k^{-1}, (d/e) y_k^{-1} )_N}
\prod_{ 1 \le i \le n}
\frac{((a q / b_i e) x_i, a q  x_i)_N} 
{((a q / b_i) x_i, (a q / e) x_i)_N} 
\nonumber\\
&& 
\quad 
\times
\sum_{\delta \in {\Bbb N}^{m} }
y_1^{\delta_1} \cdots y_m^{\delta_m}
\left( \frac{BC d^{m-1}}{a^m}\right)^{|\delta|}
q^{ - e_2 ( \delta )} \
\frac{\Delta ({y} q^{\delta})}{\Delta ({y})}
\prod_{1 \le k \le m}
\frac{1 -  (q^{-N} e/d )q^{|\delta| + \delta_k} y_k}
{1 -  (q^{-N} e/ d) y_k }
\nonumber\\
&&
\quad \quad 
\times
\prod_{1 \le k, l \le m}
\frac{(( a q / c_l^{} d ) y_k / y_l)_{\delta_k}}{(q y_k / y_l)_{\delta_k}}
\prod_{1 \le i \le n, 1 \le k \le m}
\frac{((a q /b_i^{} d) x_i y_k )_{\delta_k}}
{(( a q / d) x_i y_k )_{\delta_k}}
\nonumber\\
&& 
\quad \quad \quad 
\times
\prod_{1 \le k \le m}
\frac{1}
{((e/d) y_k, ( q^{1-N} / d ) y_k)_{\delta_k}}
\nonumber\\
&&
\quad \quad \quad \quad 
\times
{(q^{-N}, e)_{|\delta|}}
\prod_{1 \le k \le m}
\frac{( (q^{-N} e / d) y_k)_{|\delta|}}
{(( q^{-N}c_k e/ a ) y_k )_{|\delta|}}
\prod_{1 \le i \le n}
\frac{( (q^{-N} e/ a ) x_i^{-1})_{|\delta|}}
{( (q^{-N} b_i e  / a ) x_i^{-1})_{|\delta|}}.
\nonumber
\end{eqnarray}
\end{prop}

\begin{proof}
Relabel $e$ as $\mu f q^N$, i.e. interchange $ e \leftrightarrow \mu f q^N$. 
Then let  $d$  tend to infinity. Finally, by relabeling $f$ as $d$, we arrive at 
\eqref{T87T1}. 
\end{proof}

In the case when $m=1$ and $y_1= 1$, \eqref{T87T1} reduces to

\begin{eqnarray}\label{m1T87T1}
&&
\sum_{\gamma \in {\Bbb N}^{n}}
x_1^{\gamma_1} \cdots x_n^{\gamma_n}
\left( \frac{a^{2} q^{2+N }}{B c d e}\right)^{|\gamma|}
q^{ - e_2 ( \gamma )} \
\frac{\Delta({x} q^{\gamma})}{\Delta({x})}
\prod_{1 \le i \le n}
\frac{1 - a q^{ |\gamma|+ \gamma_i} x_i}
{1 - a x_i}
\\
&&
\quad 
\times
\prod_{1 \le i, j \le n}
\frac{(b_j x_i / x_j)_{\gamma_i}} {(q x_i / x_j)_{\gamma_i}}
\prod_{1 \le i \le n}
\frac{(c x_i)_{\gamma_i}}
{(( a q / d ) x_i, (a q/ e ) x_i, a q^{N+1} x_i)_{\gamma_i}}
\nonumber\\
&& 
\quad \quad 
\times
\frac
{(q^{-N}, d, e )_{|\gamma|}}
{(a q / c )_{|\gamma|}}
\prod_{1 \le i \le n}
\frac{(a x_i)_{|\gamma|}}
{( ( a q /b_i ) x_i )_{|\gamma|}}
\nonumber\\
&=&
\frac{(a q /c e, d)_N}
{(a q /c, d/e )_N}
\prod_{ 1 \le i \le n}
\frac{((a q / b_i e) x_i, a q  x_i)_N} 
{((a q / b_i) x_i, (a q / e) x_i)_N} 
\nonumber\\
&& 
\quad 
\times
{}_{2 n + 6} W_{2 n + 5}
\left[ q^{-N} e / d; \{ (a q / b_i d) x_i \}_{n}, 
\{ (q^{-N } e / a) x_i^{-1} \}_{n}, a q / c d, e, q^{-N}; q;
\frac{B c  }{a}
\right].
\nonumber
\end{eqnarray}

\begin{rem}
In the case when $m=n=1$ and $x_1 = y_1 = 1$, \eqref{T87T1} reduces to transformation formula 
for terminating ${}_8 W_7$ series
\begin{eqnarray}\label{mn1T87T1}
&& \quad \quad
{}_{8} W_{7}
\left[ a; b, c, d, e, q^{-N}; q; \frac{a^2 q^{N+2}}{b c d e}
\right]\\
&=&
\frac{(a q / b e, a q /c e, a q, d)_N}
{(a q / b, a q /c, a q / e, d/e )_N} 
{}_{8} W_{7}
\left[ q^{-N} e / d; a q / b d, q^{-N } e / a,
a q / c d, e, q^{-N}; q;
\frac{b c  }{a}
\right].
\nonumber
\end{eqnarray}
\end{rem}

Let $aq = cd $ in \eqref{m1T87T1}. Then, by rearranging parameters, we have

\begin{eqnarray}\label{RS2}
&&
\sum_{\gamma \in {\Bbb N}^{n}}
x_1^{\gamma_1} \cdots x_n^{\gamma_n}
\left( \frac{a^{} q^{1+N }}{B c}\right)^{|\gamma|}
q^{ - e_2 ( \gamma )} \
\frac{\Delta({x} q^{\gamma})}{\Delta({x})}
\prod_{1 \le i \le n}
\frac{1 - a q^{ |\gamma|+ \gamma_i} x_i}
{1 - a x_i}
\\
&&
\quad \times
\prod_{1 \le i, j \le n}
\frac{(b_j x_i / x_j)_{\gamma_i}} {(q x_i / x_j)_{\gamma_i}}
\prod_{1 \le i \le n}
\left[
{(( a q / c) x_i, a q^{1+N} x_i)_{\gamma_i}}
\right]^{-1}
\nonumber\\
&& 
\quad \quad 
\times 
\
{( q^{-N}, c)_{|\gamma|}}
\prod_{1 \le i \le n}
\frac{(a x_i)_{|\gamma|}}
{( ( a q /b_i ) x_i )_{|\gamma|}}
=
\prod_{ 1 \le i \le n}
\frac{((a q / b_i c) x_i, a q  x_i)_N} 
{((a q / b_i) x_i, (a q / c) x_i)_N}. 
\nonumber
\end{eqnarray}
\begin{rem}
In the case when $n=1$ and $x_1 = 1$, \eqref{RS2} reduces to the Rogers` summation formula 
for terminating ${}_6 W_5$ series \eqref{RogersSum1}.
\end{rem}

\medskip

As we have seen in this section, one can recover our previous results in \cite{Kaji1} from
the balanced duality transformation formula \eqref{BDT1} by limiting procedures.
Combining with results in \cite{Kaji1} and Rosengren \cite{RoseKM}, one can consider the master formula
for multiple basic hypergeometric transformations of type $A$ with different dimensions presented in 
this section to either of multiple basic Euler transformation \eqref{ETG}, Sears transformation
\eqref{ST1}, (non balanced) duality transformation \eqref{DT1} and balanced duality transformation formula
\eqref{BDT1}. 


\section{$A_n$ hypergeometric transformations}

 In this section, we present several hypergeometric transformation formulas with same dimension $n$
(for multiple basic hypergeometric series of type $A_n$) by combining some special cases of 
hypergeometric transformation with different dimensions which we have obtained in the previous section.
It contains new transformation formulas and some of these are previously known by 
Milne and his collaborators (see \cite{LM1}, \cite{MilNew1}, and \cite{MilneNagoya}).
However, our proofs of them are completely different from theirs and seem to be simpler.

\subsection{$A_n$ Watson transformations}

 In this subsection and next, we derive several $A_n$ generalization of the 
Watson transformation formula between terminating ${}_8 W_7$ series and terminating 
balanced ${}_4 \phi_3$ series 
( (2.5.1) in \cite{GR1} ):
\begin{eqnarray}\label{WatsonT1}
&& \quad  \quad 
{}_{8} W_{7} 
\left[
		a; b, c, d, e, q^{-N}; q;\frac{a^2 q^{2+N}}{b c d e}
\right]
 \\
&=& \
\frac{(a q, a q / d e)_N}{(a q / d, a q / e)_N}
{}_{4} \phi_{3} 
\left[
	\begin{matrix}
		q^{-N}, d, e, a q /b c \\
		a q / b, a q / c, d e q^{-N} / a
	\end{matrix};
q; q
\right].  
\nonumber
\end{eqnarray}

Especially, we give two types of $A_n$ Watson transformation formula
whose series in the left hand side are expressible in terms of $W^{n, 2}$ 
series here. We will use a special case ($m=1$ \eqref{m1DT1}) of the (non-balanced)
 duality transformation formula \eqref{DT1} and special cases of ${}_4 \phi_3$ series 
of type $A$ to the identities below. 
To be precise, we produce them according to the following diagram.  

\vspace{5mm}
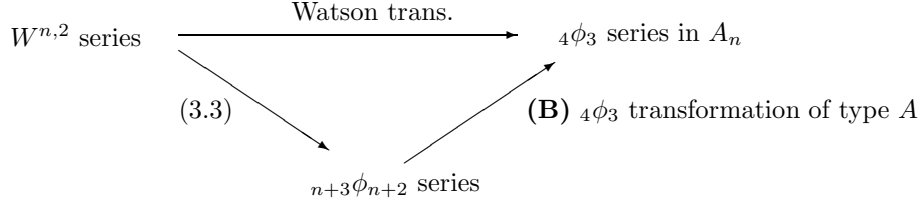
\begin{center}
\setlength{\unitlength}{1mm}
\begin{picture}(100, 30)
\put(0,20){\makebox(15,10)[r]{$W^{n, 2}$ series}}
\put(80,20){\makebox(15,10)[r]{${}_4 \phi_3$ series in $A_n$}}
\put(20, 25){\vector(1,0){45}}
\put(35, 27){Watson trans.}
\put(45, 0){\makebox(15,10)[r]{${}_{n+3} \phi_{n+2}$ series}}
\put(20, 23){\vector(3,-2){20}}
\put(50, 8){\vector(3,2){20}}
\put(20, 14){\eqref{m1DT1}}
\put(65, 14){ {\bf (B)} ${}_4 \phi_3$ transformation of type $A$}
\end{picture}
\end{center}

\medskip

\noindent
{\bf {\large The 1st one}}

\medskip

\begin{prop}

\begin{eqnarray}\label{ltAnWT1}
&&
\MW{n, 2}{\{b_i\}_n \\ \{x_i\}_n}
{a}{c, d}{e, q^{-N}}
{q; \frac{a^2 q^{N+2}}{B c d e}}\\ 
&& \quad \quad
= \
\frac{(a q / B d)_N}
{(a q / d)_N}
\prod_{1 \le i \le n}
\frac
{(a x_i)_N}
{((a q/b_i) x_i)_N}
\
\sum_{\gamma \in {\Bbb N}^n}
q^{|\gamma|} 
\frac{\Delta (x q^{\gamma})}{\Delta (x)}
\nonumber
\\
&& \quad \quad \quad \quad
\times
\frac{(q^{-N}, a q  / c e  )_{|\gamma|}}
{(B d q^{-N} /a, a q / c)_{|\gamma|}}
\prod_{1 \le i, j \le n}
\frac{(b_j x_i / x_j)_{\gamma_i}} {(q x_i / x_j)_{\gamma_i}}
\prod_{1 \le i \le n}
\frac{(d x_i )_{\gamma_i}}{( (a q/ e) x_i)_{\gamma_i}}.
\nonumber
\end{eqnarray}
\end{prop}

\begin{proof}
We combine \eqref{m1DT1}
and $n=1$ and $m \to n$ case of \eqref{ST1} 

\begin{eqnarray}\label{n1ST1}
&&
{}_{n+3} \phi_{n+2}
\left[
	\begin{matrix}
		q^{-N}, a, c,  \{u_i \}_{n} \\
		e, \{v_i  \}_{n}, a c  U q^{1-N} /  e V 
	\end{matrix}; q, q
\right]
\\
&&
\quad \quad 
=
\frac{(e,  e V / a c U)_N}
{(e/a, e V/ c U)_N}
\
\sum_{\gamma \in {\Bbb N}^n}
q^{|\gamma|} 
\frac{\Delta (v q^{\gamma})}{\Delta (v)}
\nonumber
\\
&&
\quad \quad \quad \quad
\times
\frac{(q^{-N}, a )_{|\gamma|}}
{(q^{1-N} a /e, e V / c U)_{|\gamma|}}
\prod_{1 \le i, j \le n}
\frac{( v_i / u_j)_{\gamma_i}} {(q v_i / v_j)_{\gamma_i}}
\prod_{1 \le i \le n}
\frac{( v_i/ c )_{\gamma_i}}{( v_i )_{\gamma_i}}
,
\nonumber
\end{eqnarray}
here we give in the useful form for the present proof and latter uses in this section, as {\bf (B)} 
in the above diagram. 
Note that both of the series in the right hand side of \eqref{m1DT1} and that in the left hand side 
of \eqref{n1ST1} satisfy same condition as basic hypergeometric series: namely they are terminating balanced
${}_{n+3} \phi_{n+2}$ series. On the set of variables 
\begin{eqnarray}
&&
( \{(a q / b_i e) x_i \}_{n}, a q / c e, a q / d e,  a^2 q^2 / B c d e, \{(a q / e) x_i\}_{n} ),
\end{eqnarray}
we consider the following change of variables:
\begin{eqnarray}\label{covWT1}
\tilde{a} = a q/ c e, & \tilde {c} = a q / d e, & \tilde{e} = a^2 q^2/ B c d e  
\nonumber \\
\tilde{u}_i = ( a q / b_i e) x_i,  & \tilde{v}_i = (a q/ e) x_i  \quad \quad (i= 1, \cdots , n ).&
\end{eqnarray}
For given function $\psi$, we denote by $\tilde{\psi} = \psi(a, c, e, \{ u_i \}_n, \{ v_i\}_n)$
the function that is obtained by replacing the variables
$(a, c, e, \{ u_i \}_n, \{ v_i\}_n)$ by $( \tilde{a}, \tilde{c}, \cdots).$ In this case, 
the change of variables \eqref{covWT1} is a transposition inside of each sets of numerator parameters 
in ${}_{n+3} \phi_{n+2} $ series and of denominator parameters.
Hence, the right hand side of \eqref{m1DT1} is invariant under this change of variables. 
By applying \eqref{n1ST1} to the series the right hand side in \eqref{m1DT1}, this invariance
implies \eqref{ltAnWT1}.
\end{proof}

\medskip

We also give  similar transformation formula for multiple series 
which terminates with respect to a certain multi-index. 
In this paper, we call such transformation formulas as {\it rectangular } and 
transformations for the multiple series which terminates with respect 
to the length of multi-indices as {\it triangular}.

\medskip

\begin{cor}
\begin{eqnarray}\label{itAnWT1}
&&
\MW{n, 2}{\{ q^{-m_i} \}_n \\ \{x_i\}_n}
{a}{c, d}{b, e}
{q; \frac{a^2 q^{|M| +2}}{b c d e}}\\ 
&&
\quad \quad =
\frac{(a q / b d)_{|M|}}{(a q / d)_{|M|}}
\prod_{ 1\le i \le n}
\frac{(a q x_i)_{m_i}}{((a q / b ) x_i)_{m_i}}
\sum_{\gamma \in {\Bbb N}^n}
q^{|\gamma|} 
\frac{\Delta (x q^{\gamma})}{\Delta (x)}
\nonumber\\
&& \quad \quad \quad \quad 
\times
\frac{(b , a q / c e )_{|\gamma|}}
{(b d q^{-|M|} / a,  a q /c)_{|\gamma|}}
\prod_{1 \le i, j \le n}
\frac{(q^{- m_j}x_i / x_j)_{\gamma_i}} {(q x_i / x_j)_{\gamma_i}}
\prod_{1 \le i \le n}
\frac{(d x_i )_{\gamma_i}}{((a q / e) x_i)_{\gamma_i}}
.
\nonumber
\end{eqnarray}
\end{cor}

\begin{proof}
We first write the product factor in the right hand side of \eqref{ltAnWT1}
as a quotient of infinite products using \eqref{ShiftFact}.
Set $b_i = q^{-m_i}$ in \eqref{ltAnWT1}, and notice that \eqref{itAnWT1} is 
true for $b= q^{-N}$ for all nonnegative integer $N$. Clear the denominators 
in \eqref{itAnWT1}. Then we find that it is a polynomial identity in $b^{-1}$ with an infinite 
number of roots. Thus, \eqref{itAnWT1} is true for arbitrary $b$.
\end{proof}

All the corollaries in this section can be proved by similar arguments from 
the formulas in the preceding propositions. So, hereafter we will not repeat 
this procedure in the rest of this paper.  

\begin{rem}
\eqref{itAnWT1} has appeared in  Theorem 6.1 of Milne and Lilly \cite{LM1}.  
\end{rem}

\medskip

\noindent
{\bf {\large The 2nd  one}}

\medskip

\begin{prop}

\begin{eqnarray}\label{ltAnWT2}
&&
\MW{n, 2}{\{b_i\}_n \\ \{x_i\}_n}
{a}{c, d}{e, q^{-N}}
{q; \frac{a^2 q^{N+2}}{B c d e}}\\ 
&& \quad 
=
\prod_{1 \le i \le n}
\frac
{(a x_i, (a q / b_i e) x_i)_N}
{((a q / b_i) x_i, ( a q / e ) x_i)_N}
\sum_{\gamma \in {\Bbb N}^n}
q^{|\gamma|} 
\frac{\Delta (z q^{\gamma})}{\Delta (z)}
\nonumber
\\
&& 
\quad \quad 
\times 
\frac{(q^{-N}, e )_{|\gamma|}}
{(a q / c, a q/ d)_{|\gamma|}}
\prod_{1 \le i, j \le n}
\frac{(b_j z_i / z_j)_{\gamma_i}} {(q z_i / z_j)_{\gamma_i}}
\prod_{1 \le i \le n}
\frac{((a q / c d ) z_i )_{\gamma_i}}{( (B e q^{-N}/ a) z_i)_{\gamma_i}}
,
\nonumber
\end{eqnarray}
where $z_i = b_i / B x_i$ for $1 \le i  \le n$.
\end{prop}

\begin{proof}
We combine \eqref{m1DT1}
and $n=1$ and $m \to n$ case of \eqref{ST2} 
\begin{eqnarray}\label{n1ST2}
&&
{}_{n+3} \phi_{n+2}
\left[
	\begin{matrix}
		q^{-N}, a, c, \{u_i \}_{n}\\
		e,  \{v_i \}_{n}, a c U  q^{1-N} /e  V 
	 \end{matrix}; q, q
\right]
\\
&&
\quad \quad 
=
\frac{(e V /a U, e V /c U)_N}
{(e V / a c U, e)_N}
\prod_{1 \le i \le n}
\frac{(u_i)_N}{(v_i)_N}
\
\sum_{\gamma \in {\Bbb N}^n}
q^{|\gamma|} 
\frac{\Delta (u^{-1} q^{\gamma})}{\Delta (u^{-1})}
\nonumber\\
&&
\quad \quad \quad \quad 
\times
\frac{(q^{-N}, e V /a c U )_{|\gamma|}}
{(e V /a U, e V /c U)_{|\gamma|}}
\prod_{1 \le i, j \le n}
\frac{(v_j/ u_i)_{\gamma_i}} {(q u_j / u_i)_{\gamma_i}}
\prod_{1 \le i \le n}
\frac{(e/ u_i)_{\gamma_i}}{(q^{1-N} / u_i)_{\gamma_i}}, 
\nonumber
\end{eqnarray}
here we present in a modified form, as {\bf(B)}. 
Notice that both of the series in the right hand side of \eqref{m1DT1} and that in the left hand 
side of \eqref{n1ST2} are terminating balanced ${}_{n+3} \phi_{n+2}$ series. In this case, we consider 
the following change of variables
\begin{eqnarray}\label{covWT2}
\tilde{a} = a q / c e, & \tilde {c} = a q / d e, & \tilde{e} = a^2 q^2/ B c d e  
\nonumber \\
\tilde{u}_i = ( a q / b_i e) x_i,  & \tilde{v}_i = ( a q / e) x_i  \quad \quad (i= 1, \cdots , n ).&
\end{eqnarray}
Since this change of variables is  same as \eqref{covWT1} in the proof of Proposition 4.1. 
one can obtain the desired identity \eqref{ltAnWT2} by plugging \eqref{n1ST2} to the series in 
the right hand side of \eqref{m1DT1} according to this change of variables. 
\end{proof}

\medskip

\noindent
{\bf Rectangular version}

\medskip

\begin{cor}
\begin{eqnarray}\label{itAnWT2}
&&
\MW{n, 2}{\{ q^{-m_i} \}_n \\ \{x_i\}_n}
{a}{c, d}{b, e}
{q; \frac{a^2 q^{|M| +2}}{b c d e}}\\ 
&& \quad = 
\prod_{ 1\le i \le n}
\frac{(a q x_i, (a q /b e) x_i)_{m_i}}{((a q / b ) x_i, (a q / e) x_i)_{m_i}}
\sum_{\gamma \in {\Bbb N}^n}
q^{|\gamma|} 
\frac{\Delta (z q^{\gamma})}{\Delta (z)}
\nonumber\\
&& \quad \quad \times
\frac{(b , e)_{|\gamma|}}
{(a q /c, a q/ d)_{|\gamma|}}
\prod_{1 \le i, j \le n}
\frac{(q^{- m_j} z_i / z_j)_{\gamma_i}} {(q z_i / z_j)_{\gamma_i}}
\prod_{1 \le i \le n}
\frac{( (a q / c e) x_i)_{\gamma_i}}{(( be q^{-|M|} / a) x_i )_{\gamma_i}}
,
\nonumber
\end{eqnarray}
where $z_i = q^{-m_i+|M|} x_i^{-1}$ for $i = 1, \cdots , n$. 
\end{cor}


\begin{rem}
In the case when $n=1$ and $x_1 = 1$, all of \eqref{ltAnWT1}, \eqref{itAnWT1}, \eqref{ltAnWT2} 
and \eqref{itAnWT2} reduce to the Watson transformation formula \eqref{WatsonT1}. 
\eqref{ltAnWT2} can be proved by a similar limiting procedure as the previous section from
the following $A_n$ Bailey transformation formula for terminating balanced ${}_{10} W_9$ series:

\begin{eqnarray}\label{Masatoshi-san}
&&
\MW{n,3}
{ \{ e_i\}_n \\ \{x_i\}_n }
{{a}}{b, c, d}
{q^{-N}, f, a \lambda q^{1 + N} /  E f }{q}
 \\
&& \quad  = 
\prod_{1 \le i \le n}
\frac
{(a q x_i, (a q / e_i f) x_i, (\lambda q / e_i) z_i, (\lambda q / f) z_i)_N}
{((a q / e_i) x_i, (a q / f) x_i, \lambda q  z_i, (\lambda q / e_i f) z_i)_N  }
\nonumber
\\
&& \quad \quad \quad \times
\MW{n,3}
{\{e_i\}_n \\ \{z_i\}_n}
{\lambda}{a q / c d, a q / b d, a q / b c}
{q^{-N}, f, a \lambda q^{1 + N} / E f }{q}
\nonumber
\end{eqnarray}
where
$\lambda = a^2 q / b c d$ and $z_i = e_i / E x_i$ for $1 \le i \le n$,
which has first appeared as (4.36) in \cite{KajiNou}, and 
by rearranging the parameters.

\end{rem}

\subsection{Another type of $A_n$ Watson transformation}

Here, we present a yet another $A_n$ Watson transformation with 
a different form in the series in both sides from those in the previous subsection. 
We will use the $m=1$ case of the terminating ${}_8 W_7$ transformation \eqref{m1T87T1} of type 
$A$ in Section 3.3. and the $n=1$ case of the duality transformation formula 
in a modified form to produce. 
We construct it according to the following procedure:
\vspace{5mm}
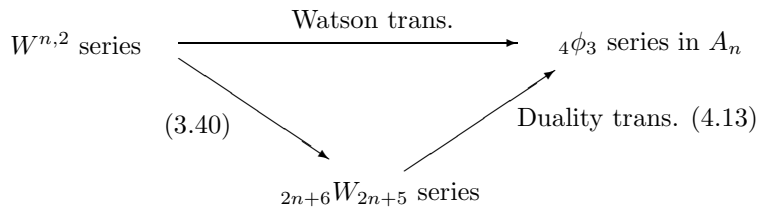
\begin{center}
\setlength{\unitlength}{1mm}
\begin{picture}(100, 30)
\put(0,20){\makebox(15,10)[r]{$W^{n, 2}$ series}}
\put(80,20){\makebox(15,10)[r]{${}_4 \phi_3$ series in $A_n$}}
\put(20, 25){\vector(1,0){45}}
\put(35, 27){Watson trans.}
\put(45, 0){\makebox(15,10)[r]{${}_{2n+6} W_{2n+5}$ series}}
\put(20, 23){\vector(3,-2){20}}
\put(50, 8){\vector(3,2){20}}
\put(7, 9){\makebox(20, 10)[r]{\eqref{m1T87T1}}}
\put(65, 14){Duality trans. \eqref{rn1DT1}}
\end{picture}
\end{center}

\begin{prop}
\begin{eqnarray}\label{ltAnWT3}
&&
\sum_{\gamma \in {\Bbb N}^{n}}
x_1^{\gamma_1} \cdots x_n^{\gamma_n}
\left( \frac{a^{2} q^{2+N }}{B c d e}\right)^{|\gamma|}
q^{ - e_2 ( \gamma )} \ 
\frac{\Delta({x} q^{\gamma})}{\Delta({x})}
\prod_{1 \le i \le n}
\frac{1 - a q^{ |\gamma|+ \gamma_i} x_i}
{1 - a x_i}
\\
&& \quad 
\times
\prod_{1 \le i, j \le n}
\frac{(b_j x_i / x_j)_{\gamma_i}} {(q x_i / x_j)_{\gamma_i}}
\prod_{1 \le i \le n}
\frac{(c x_i)_{\gamma_i}}
{( (a q / d) x_i,  (a q/ e) x_i, a q^{N+1} x_i)_{\gamma_i}}
\nonumber\\
&& \quad \quad \quad 
\times
\frac
{(q^{-N}, d, e )_{|\gamma|}}
{(a q / c )_{|\gamma|}}
\prod_{1 \le i \le n}
\frac{(a x_i )_{|\gamma|}}
{( ( a q /b_i ) x_i)_{|\gamma|}}
\nonumber\\
&=&
\frac{(a q / B c )_N}
{(a q /c)_N}
\prod_{ 1 \le i \le n}
\frac{(a q  x_i )_N} 
{((a q / b_i) x_i)_N} 
\
\sum_{\delta \in {\Bbb N}^{n} }
q^{|\delta|}
\frac{\Delta ({x} q^{\delta})}{\Delta ({x})}
\nonumber\\
&& \quad \quad \quad 
\times
\frac{(q^{-N})_{|\delta|}}
{(q^{-N} B c / a )_{|\delta|}}
\prod_{1 \le i, j  \le n}
\frac{( b_j x_i / x_j)_{\delta_i}}{(q x_i / x_j)_{\delta_i}}
\prod_{1 \le i \le n}
\frac{((a q / d e ) x_i, c x_i)_{\delta_i}}
{(( a q / d) x_i, ( a q / e) x_i)_{\delta_i}}
.
\nonumber
\end{eqnarray}
\end{prop}

\begin{proof}
We use \eqref{m1T87T1}
and \eqref{n1DT1} with a modified form 
\begin{eqnarray}\label{rn1DT1}
&&{}_{2 n + 6} W_{2 n + 5}
\left[	
a; b, \{u_i \}_{n}, 
d, \{ v_i \}_{n}, q^{-N}; q; 
\frac{a^{n+1} q^{N+n+1}}{b d U V}
\right]
\\
&&
\quad \quad 
=
\frac
{( a^{n+1} q^{n+1} / b d U V, a q)_N}
{(a q / b, a q / d)_N }
\prod_{ 1 \le i \le n}
\frac
{(v_i)_N}
{(a q / u_i )_N}
\
\sum_{\delta \in {\Bbb N}^{n} }
q^{|\delta|}
\frac{\Delta ({v}^{-1} q^{\delta})}{\Delta ({v^{-1}})}
\nonumber\\
&&
\quad \quad \quad \quad 
\times
\frac{(q^{-N})_{|\delta|}}
{( a^{n+1} q^{n+1} / b d U V )_{|\delta|}}
\prod_{1 \le i,j  \le n}
\frac{(a q / u_j v_i)_{\delta_i}}{(q v_j/ v_i)_{\delta_i}}
\prod_{1 \le i \le n}
\frac{(a q /b v_i, a q / d v_i)_{\delta_i}}
{(a q / v_i, q^{1-N} / v_i)_{\delta_i}}
.
\nonumber
\end{eqnarray}
to obtain.  
It is not hard to see that both of the series in the right hand side of 
\eqref{m1T87T1} and in the left hand side of \eqref{rn1DT1}
satisfy the same condition: they are 
${}_{2n+6} W_{2n+5}$ series and very-well-poised-balanced. 
We consider the following change of variables:

\begin{eqnarray} \label{covWT3}
\tilde{a} = q^{-N} e / d,  & \tilde{b} = a q / c d, &
\tilde{d} = e,
\\
\tilde{u}_i = (a q/ b_i d)  x_i, & \tilde{v}_i = (q^{-N} e/ a) x_i^{-1}  \quad \quad 
(i= 1, \cdots , n), &
\nonumber
\end{eqnarray}
which is a transposition of the variables in 
${}_{2n + 6} W_{2n+5} $ series in \eqref{m1T87T1}.
Note that ${}_{r+3} W_{r+2} $ series is symmetric with respect to 
the variables $a_1, \cdots , a_r$.
So the series in the right hand side of \eqref{m1T87T1} is invariant 
under this change of variables. This invariance implies \eqref{ltAnWT3}
by applying \eqref{rn1DT1} according to the change of variables \eqref{covWT3}.
\end{proof}

\medskip

\noindent
{\bf Rectangular version}

\medskip

\begin{cor}
\begin{eqnarray}\label{itAnWT3}
&&
\sum_{\gamma \in {\Bbb N}^{n}}
x_1^{\gamma_1} \cdots x_n^{\gamma_n}
\left( \frac{a^{2} q^{2+|M| }}{b c d e}\right)^{|\gamma|}
q^{ - e_2 ( \gamma )} \
\frac{\Delta({x} q^{\gamma})}{\Delta({x})}
\prod_{1 \le i \le n}
\frac{1 - a q^{ |\gamma|+ \gamma_i} x_i }
{1 - a x_i}
\\
&& \quad 
\times
\prod_{1 \le i, j \le n}
\frac{(q^{-m_j}  x_i / x_j)_{\gamma_i}} {(q x_i / x_j)_{\gamma_i}}
\prod_{1 \le i \le n}
\frac{(c x_i)_{\gamma_i}}
{( ( a q /b) x_i, (a q / d) x_i, (a q/ e) x_i)_{\gamma_i}}
\nonumber\\
&&
\quad \quad \quad 
\times
\frac
{( b, d, e )_{|\gamma|}}
{(a q / c )_{|\gamma|}}
\prod_{1 \le i \le n}
\frac{(a x_i)_{|\gamma|}}
{( ( a q^{m_i} x_i)_{|\gamma|}}
\nonumber\\
&=&
\frac{(a q / b c )_{|M|}}
{(a q /c)_{|M|}}
\prod_{ 1 \le i \le n}
\frac{(a q  x_i)_{m_i}} 
{((a q / b) x_i)_{m_i}} 
\
\sum_{\delta \in {\Bbb N}^{n} }
q^{|\delta|}
\frac{\Delta ({x} q^{\delta})}{\Delta ({x})}
\nonumber
\\
&& \quad \quad 
\times
\frac{(b)_{|\delta|}}
{(q^{-|M|} b c / a )_{|\delta|}}
\prod_{1 \le i, j  \le n}
\frac{( q^{-m_j} x_i / x_j)_{\delta_i}}{(q x_i / x_j)_{\delta_i}}
\prod_{1 \le i \le n}
\frac{((a q / d e  ) x_i, c x_i)_{\delta_i}}
{(( a q / d) x_i,  ( a q / e) x_i)_{\delta_i}}
\nonumber
\end{eqnarray}
\end{cor}

\begin{rem}
In the case when $n=1$ and $x_1 = 1$,  \eqref{ltAnWT3} and \eqref{itAnWT3} reduce to the 
Watson transformation formula \eqref{WatsonT1}. 
\end{rem}

\begin{rem}
In a similar fashion as we yield \eqref{ltAnWT3},  we also obtain 
the following transformation formula between $A_n$ ${}_8 W_7$ series
and $A_n$ terminating balanced ${}_4 \phi_3$ series: 
\begin{eqnarray}\label{ltAnWT4}
&&
\sum_{\gamma \in {\Bbb N}^{n}}
x_1^{\gamma_1} \cdots x_n^{\gamma_n}
\left( \frac{a^{2} q^{2+N }}{B c d e}\right)^{|\gamma|}
q^{ - e_2 ( \gamma )} \ 
\frac{\Delta({x} q^{\gamma})}{\Delta({x})}
\prod_{1 \le i \le n}
\frac{1 - a q^{ |\gamma|+ \gamma_i} x_i}
{1 - a x_i}
\\
&& \quad 
\times
\prod_{1 \le i, j \le n}
\frac{(b_j x_i / x_j)_{\gamma_i}} {(q x_i / x_j)_{\gamma_i}}
\prod_{1 \le i \le n}
\frac{(c x_i)_{\gamma_i}}
{( (a q / d) x_i,  (a q/ e) x_i, a q^{N+1} x_i)_{\gamma_i}}
\nonumber\\
&& \quad \quad \quad 
\times
\frac
{(q^{-N}, d, e )_{|\gamma|}}
{(a q / c )_{|\gamma|}}
\prod_{1 \le i \le n}
\frac{(a x_i )_{|\gamma|}}
{( ( a q /b_i ) x_i)_{|\gamma|}}
\nonumber\\
&=& B^N
\frac{(a q / B c )_N}
{(a q /c)_N}
\prod_{ 1 \le i \le n}
\frac{(a q  x_i,  (a q / b_i d) x_i, (a q / b_i e) x_i)_N} 
{((a q / b_i, (a q / d) x_i, (a q / e) x_i) x_i)_N} 
\nonumber\\
&& 
\quad  
\times
\sum_{\delta \in {\Bbb N}^{n} }
q^{|\delta|}
\frac{\Delta ({x^{-1}} q^{\delta})}{\Delta ({x^{-1}})}
\prod_{1 \le i, j  \le n}
\frac{( b_j x_j / x_i)_{\delta_i}}{(q x_j / x_i)_{\delta_i}}
\nonumber\\
&& \quad \quad \quad 
\times
\frac{(q^{-N})_{|\delta|}}
{(q^{-N} B c / a )_{|\delta|}}
\prod_{1 \le i \le n}
\frac{((q^{-1-N} b_i c d e / a^2 ) x_i^{-1}, (q^{-N} b_i / a ) x_i^{-1})_{\delta_i}}
{(( q^{-N} b_i d / a) x_i^{-1}, ( q^{-N} b_i e / a) x_i^{-1} )_{\delta_i}}
.
\nonumber
\end{eqnarray}
We obtain it by applying the change of variables:
\begin{eqnarray} \label{covWT4}
\tilde{a} = q^{-N} e / d,  & \tilde{b} = a q / c d, 
& \tilde{d} = e,
\\
\tilde{u}_i = (q^{-N} e/ a) x_i^{-1}, \quad &
\tilde{v}_i = (a q/ b_i d)  x_i, &
  \quad \quad 
(i= 1, \cdots , n). 
\nonumber
\end{eqnarray}  
The rectangular version of \eqref{ltAnWT4} is given by
\begin{eqnarray}\label{itAnWT4}
&&
\sum_{\gamma \in {\Bbb N}^{n}}
x_1^{\gamma_1} \cdots x_n^{\gamma_n}
\left( \frac{a^{2} q^{2+|M| }}{b c d e}\right)^{|\gamma|}
q^{ - e_2 ( \gamma )} \ 
\frac{\Delta({x} q^{\gamma})}{\Delta({x})}
\prod_{1 \le i \le n}
\frac{1 - a q^{ |\gamma|+ \gamma_i} x_i}
{1 - a x_i}
\\
&& \quad 
\times
\prod_{1 \le i, j \le n}
\frac{(q^{- m_j} x_i / x_j)_{\gamma_i}} {(q x_i / x_j)_{\gamma_i}}
\prod_{1 \le i \le n}
\frac{(c x_i)_{\gamma_i}}
{( (a q / b) x_i,  (a q/ d) x_i, (a q / e) x_i)_{\gamma_i}}
\nonumber\\
&& \quad \quad \quad 
\times
\frac
{(b , d, e )_{|\gamma|}}
{(a q / c )_{|\gamma|}}
\prod_{1 \le i \le n}
\frac{(a x_i )_{|\gamma|}}
{( ( a q^{1+ m~i} ) x_i)_{|\gamma|}}
\nonumber\\
&=& 
b^{|M|}
\frac{(a q / b c )_{|M|}}
{(a q /c)_{|M|}}
\prod_{ 1 \le i \le n}
\frac{(a q  x_i,  (a q / b d) x_i, (a q / b e) x_i)_{m_i}} 
{((a q / b) x_i, (a q / d) x_i, (a q / e) x_i) x_i)_{m_i}} 
\nonumber\\
&& 
\quad  
\times
\sum_{\delta \in {\Bbb N}^{n} }
q^{|\delta|}
\frac{\Delta ({x^{-1}} q^{\delta})}{\Delta ({x^{-1}})}
\prod_{1 \le i, j  \le n} 
\frac{( q^{-m_j} x_j / x_i)_{\delta_i}}{(q x_j / x_i)_{\delta_i}}
\nonumber\\
&& \quad \quad \quad 
\times
\frac{(b)_{|\delta|}}
{(q^{-|M|} b c / a )_{|\delta|}}
\prod_{1 \le i \le n}
\frac{((q^{-1- m_i} b c d e / a^2 ) x_i^{-1}, (q^{-m_i} b / a ) x_i^{-1})_{\delta_i}}
{(( q^{-m_i} b d / a) x_i^{-1}, ( q^{-m_i} b e / a) x_i^{-1} )_{\delta_i}}
.
\nonumber
\end{eqnarray}
In the case when $n=1$ and $x_1 = 1$, the transformation \eqref{ltAnWT4} and  \eqref{itAnWT4}
reduces to the following transformation formula between terminating ${}_8 W_7$ series and 
terminating balanced ${}_4 \phi_3$ series:
\begin{eqnarray}\label{WatsonT2}
&& \quad  \quad 
{}_{8} W_{7} 
\left[
		a; b, c, d, e, q^{-N}; q;\frac{a^2 q^{2+N}}{b c d e}
\right]
 \\
&=& \
b^N
\frac{(a q, a q / b c, a q / b d, a q / b e )_N}{(a q / b,  a q / c, a q / d, a q / e)_N}
{}_{4} \phi_{3} 
\left[
	\begin{matrix}
		q^{-N}, b, q^{-1-N} b c d e / a^2, q^{-N} b / a \\
		b c q^{-N} / a, b d q^{-N} / a, b e q^{-N} / a, 
	\end{matrix};
q; q
\right].  
\nonumber
\end{eqnarray}
\end{rem}


\subsection{$A_n$ Sears transformations}

 In this and next subsection, we present some $A_n$ generalizations
of the Sears transformation formula for terminating balanced 
${}_4 \phi_3 $ series  \eqref{SearsT1}.
In particular, we will prove two $A_n$ Sears transformations whose form of the series 
in both sides are same as that in the right hand side of the $A_n$ Watson transformation
formulas in Section 4.1.  We produce these identities by combining certain special cases 
of ${}_4 \phi_3$ series of type $A$ in Section 3.2. Our way to prove them is figured as the 
following diagram:

\vspace{5mm}
\begin{center}
\setlength{\unitlength}{1mm}
\begin{picture}(100, 30)
\put(0,20){\makebox(15,10)[r]{${}_4 \phi_3$ series in $A_n$}}
\put(80,20){\makebox(15,10)[r]{${}_4 \phi_3$ series in $A_n$}}
\put(20, 25){\vector(1,0){45}}
\put(35, 27){$A_n$ Sears trans.}
\put(45, 0){\makebox(15,10)[r]{${}_{n+3} \phi_{n+2}$ series}}
\put(20, 23){\vector(3,-2){20}}
\put(50, 8){\vector(3,2){20}}
\put(-25, 14){${}_4 \phi_3$ transformation of type $A$ {\bf (A)}}
\put(65, 14){ {\bf (B)} ${}_4 \phi_3$ transformation of type $A$}
\end{picture}
\end{center}

\medskip

\noindent
{\bf The 1st one }

\medskip

\begin{prop}
\begin{eqnarray}\label{ltAnST1}
&&
\sum_{\gamma \in {\Bbb N}^n}
q^{|\gamma|} 
\frac{\Delta (x q^{\gamma})}{\Delta (x)}
\prod_{1 \le i, j \le n}
\frac{(b_j x_i / x_j)_{\gamma_i}} {(q x_i / x_j)_{\gamma_i}}
\prod_{1 \le i \le n}
\frac{(c x_i )_{\gamma_i}}{(d x_i)_{\gamma_i}}
\\
&& \quad \times
\frac{(q^{-N}, a)_{|\gamma|}}
{(e, a B c q^{1-N} / d e)_{|\gamma|}}
=
\frac{(e / B, d e /a c)_N}
{(e, d e / a B c)_N} 
\
\sum_{\delta \in {\Bbb N}^n}
q^{|\delta|} 
\frac{\Delta (x q^{\delta})}{\Delta (x)}
\nonumber\\
&& \quad \quad \quad 
\times
\frac{(q^{-N}, d/c)_{|\delta|}}
{(d e/a c, q^{1-N} B / e)_{|\delta|}}
\prod_{1 \le i, j \le n}
\frac{(b_j x_i / x_j)_{\delta_i}} {(q x_i / x_j)_{\delta_i}}
\prod_{1 \le i \le n}
\frac{((d/a) x_i)_{\delta_i}}{(d x_i)_{\delta_i}}
.
\nonumber
\end{eqnarray}
\end{prop}

\begin{proof}
We use \eqref{m1ST2} as {\bf (A)} and \eqref{n1ST2} as {\bf (B)} in the above diagram. 
It is not hard to see that both of the series in the right hand side of \eqref{m1ST2} 
and in the left hand side in \eqref{n1ST2} are terminating balanced ${}_{n+3} \phi_{n+2}$ 
series. We consider the following change of variables:

\begin{eqnarray}\label{1covST1}
\tilde {a} = e/a, &
\tilde{c} = f/a, & \tilde{e} = d e f /a^2 B c. 
\\
\tilde{u_i} = \frac{e f}{a b_i c} z_i, & 
{\displaystyle \tilde{v_i} = \frac{e f}{a c}z_i^{}} \qquad (1 \le i \le n). &
\nonumber 
\end{eqnarray}
Note that the series in the right hand side of \eqref{m1ST2} is invariant under this change 
of variables. Applying \eqref{n1ST2} to the ${}_{n+3} \phi_{n+2}$ series in the right hand side 
in \eqref{m1ST2} leads to the desired result \eqref{ltAnST1}. 
\end{proof}

 In \cite{KajiS}, we have already shown that \eqref{ltAnST1} can also be obtained by combining 
\eqref{m1ST1} as {\bf (A)} and \eqref{n1ST1} as {\bf(B)}. In this case, The change of variables is given by
\begin{eqnarray}\label{2covST1}
\tilde {a} = d/ c, &
\tilde{c} = a, & \tilde{e} = d e / B c. 
\\
\tilde{u_i} = \frac{d}{b_i} x_i, & 
{\displaystyle \tilde{v_i} = d x_i^{}} \qquad (1 \le i \le n). &
\nonumber 
\end{eqnarray}

\medskip 

\noindent
{\bf Rectangular version } 

\medskip

\begin{cor}
\begin{eqnarray}\label{itAnST1}
&&
\sum_{\gamma \in {\Bbb N}^n}
q^{|\gamma|} 
\frac{\Delta (x q^{\gamma})}{\Delta (x)}
\prod_{1 \le i, j \le n}
\frac{(q^{-m_j} x_i / x_j)_{\gamma_i}} {(q x_i / x_j)_{\gamma_i}}
\prod_{1 \le i \le n}
\frac{(c x_i )_{\gamma_i}}{(d x_i)_{\gamma_i}}
\\
&&  \quad 
\times
\frac{( a, b )_{|\gamma|}}
{(e, a b c q^{1-|M|} / d e)_{|\gamma|}}
\ = \
\frac{(e / b, d e /a c)_{|M|}}
{(e, d e / a b c)_{|M|}} 
\sum_{\delta \in {\Bbb N}^n}
q^{|\delta|} 
\frac{\Delta (x q^{\delta})}{\Delta (x)}
\nonumber\\
&& \quad \quad \quad 
\times
\frac{(b, d/c)_{|\delta|}}
{(d e/a c, q^{1-|M|} b / e)_{|\delta|}}
\prod_{1 \le i, j \le n}
\frac{(q^{-m_j} x_i / x_j)_{\delta_i}} {(q x_i / x_j)_{\delta_i}}
\prod_{1 \le i \le n}
\frac{((d/a) x_i )_{\delta_i}}{(d x_i )_{\delta_i}}
.
\nonumber
\end{eqnarray}
\end{cor}


\medskip

\noindent
{\bf The 2nd one }

\medskip

\begin{prop}
\begin{eqnarray}\label{ltAnST2}
&&
\sum_{\gamma \in {\Bbb N}^n}
q^{|\gamma|} 
\frac{\Delta (x q^{\gamma})}{\Delta (x)}
\prod_{1 \le i, j \le n}
\frac{(b_j x_i / x_j)_{\gamma_i}} {(q x_i / x_j)_{\gamma_i}}
\prod_{1 \le i \le n}
\frac{(c x_i)_{\gamma_i}}{(d x_i)_{\gamma_i}}
\\
&& \quad 
\times
\frac{(q^{-N}, a)_{|\gamma|}}
{(e, a B c q^{1-N} / d e)_{|\gamma|}}
\ = \
\frac{(d e /a c)_N}
{( d e / a B c)_N} 
\prod_{1 \le i \le n}
\frac{((d/ b_i) x_i)_N}
{(d x_i)_N}
\nonumber
\\
&& \quad \quad \quad 
\times \
\sum_{\delta \in {\Bbb N}^n}
q^{|\delta|} 
\frac{\Delta (z q^{\delta})}{\Delta (z)}
\frac{(q^{-N}, e/a)_{|\delta|}}
{(d e/a c, e)_{|\delta|}}
\prod_{1 \le i, j \le n}
\frac{(b_j z_i / z_j)_{\delta_i}} {(q z_i / z_j)_{\delta_i}}
\prod_{1 \le i \le n}
\frac{((e/c) z_i)_{\delta_i}}{((q^{1-N} B / d) z_i)_{\delta_i}}
,
\nonumber
\end{eqnarray}
where $z_i = b_i / B x_i$ for $ 1 \le i \le n$.
\end{prop}

\begin{proof}
We use \eqref{m1ST1}
as {\bf (A)} and \eqref{n1ST2} as {\bf(B)}.
In this case, the change of variables is given by 
\begin{eqnarray}
\tilde{a} = a, &
\tilde{c} = d/ c, & \tilde{e} = d e /B c, \\
\tilde{u_i} = \frac{d}{b_i} x_i , & \tilde{v_i} = d x_i \qquad (1 \le i \le n). & 
\nonumber 
\end{eqnarray}
For the rest, one can easily verify in a similar way as the proof of Proposition 4.4.
\end{proof}

Note that \eqref{ltAnST2} can also be obtained by combining \eqref{m1ST2} as {\bf (A)} and \eqref{n1ST1} as {\bf (B)}. 
In this case, the change of variables is given by
\begin{eqnarray}
\tilde{a} = e/a, &
\tilde{c} = f/ a, & \tilde{e} = e f / a B, \\
\tilde{u_i} = \frac{e f}{a b_i c} z_i , & \tilde{v_i} = {\displaystyle \frac{e f}{a c} z_i} 
 \qquad (1 \le i \le n). & 
\nonumber 
\end{eqnarray}

\medskip

\noindent
{\bf Rectangular version}

\medskip

\begin{cor}
\begin{eqnarray}\label{itAnST2}
&&
\sum_{\gamma \in {\Bbb N}^n}
q^{|\gamma|} 
\frac{\Delta (x q^{\gamma})}{\Delta (x)}
\prod_{1 \le i, j \le n}
\frac{(q^{- m_j} x_i / x_j)_{\gamma_i}} {(q x_i / x_j)_{\gamma_i}}
\prod_{1 \le i \le n}
\frac{(c x_i )_{\gamma_i}}{(d x_i )_{\gamma_i}}
\\
&& 
\quad 
\times \
\frac{(a, b)_{|\gamma|}}
{(e, a b c q^{1-|M|} / d e)_{|\gamma|}}
\ = \
\frac{(d e /a c)_{|M|}}
{( d e / a b c)_{|M|}}
\prod_{ 1 \le i \le n}
\frac{((d/b) x_i)_{m_i}}
{(d x_i)_{m_i}} 
\nonumber\\
&&
\quad \quad \quad
\times \
\sum_{\delta \in {\Bbb N}^n}
q^{|\delta|} 
\frac{\Delta (z q^{\delta})}{\Delta (z)}
\frac{(e/a, b)_{|\delta|}}
{(d e/a c, e)_{|\delta|}}
\prod_{1 \le i, j \le n}
\frac{(q^{-m_j} z_i / z_j)_{\delta_i}} {(q z_i / z_j)_{\delta_i}}
\prod_{1 \le i \le n}
\frac{((e/c) z_i )_{\delta_i}}{((b q^{1-|M|}/d) z_i)_{\delta_i}}
.
\nonumber
\end{eqnarray}
where $ z_i = q^{-m_i + |M| } x_i^{-1}$ for $ 1\le i \le n$.
\end{cor}

\begin{rem}
In the case when $n=1$ and $x_1 = 1$, all of $A_n$ terminating balanced ${}_4 \phi_3$ transformation formulas 
\eqref{ltAnST1}, \eqref{itAnST1}, \eqref{ltAnST2} and \eqref{itAnST2}
reduce to  the Sears transformation \eqref{SearsT1}.
Especially, \eqref{ltAnST1} has already appeared in our previous work \cite{KajiS} and
\eqref{itAnST2} is originally due to Milne and Lilly (Theorem 6.5 in \cite{LM1}).
Note also that \eqref{ltAnST1} can be proved by duplicating \eqref{ltAnST2}
\end{rem}

\begin{rem}
Two  $A_n$ Watson transformations \eqref{ltAnWT1} and \eqref{ltAnWT2} transpose to one another by 
$A_n$ Sears transformation formula \eqref{ltAnST2}. 
Similarly,
two  $A_n$ Watson transformations \eqref{itAnWT1} and \eqref{itAnWT2} transpose to one another by 
$A_n$ Sears transformation formula \eqref{itAnST2}. 

\end{rem} 


\subsection{Another type of $A_n$ Sears transformation}

Here, we give a yet another type of  $A_n$ Sears transformation formula whose form of the series in both sides 
are the same as that in the right hand side of $A_n$ Watson transformation formulas in Section 4.2. 
We use the $m=1$ case \eqref{m1IDT1} of the inversion of the duality transformation 
\eqref{IDT1} and a certain special case \eqref{rn1DT1} of the duality transformation formula \eqref{DT1}.
Our road map is as follows:
 
\vspace{5mm}
\begin{center}
\setlength{\unitlength}{1mm}
\begin{picture}(100, 30)
\put(0,20){\makebox(15,10)[r]{${}_4 \phi_3$ series in $A_n$}}
\put(80,20){\makebox(15,10)[r]{${}_4 \phi_3$ series in $A_n$}}
\put(20, 25){\vector(1,0){45}}
\put(23, 27){$A_n$ Sears trans.}
\put(0,0){\makebox(15,10)[r]{${}_{2n+6} W_{2n+5}$ series}}
\put(80,0){\makebox(15,10)[r]{${}_{2n+6} W_{2n+5}$ series}}
\put(20, 5){\vector(1,0){45}}
\put(30, 7){transposition}
\put(5, 20){\vector(0, -1){10}}
\put(75, 10){\vector(0, 1){10}}
\put(8, 15){\eqref{m1IDT1}}
\put(78, 15){\eqref{rn1DT1} }
\end{picture}
\end{center}

\begin{prop}
\begin{eqnarray}\label{ltAnST3}
&&
\sum_{\gamma \in {\Bbb N}^{n}}
q^{|\gamma|} 
\frac{\Delta({x} q^{\gamma})}{\Delta({x})}
\prod_{1 \le i, j \le n}
\frac{(a_j x_i / x_j)_{\gamma_i}} {(q x_i / x_j)_{\gamma_i}}
\prod_{1 \le i \le n}
\frac{(b x_i, c x_i )_{\gamma_i}} 
{(e x_i,
( A b c q^{1-N}/ d e) x_i)_{\gamma_i}}
\\
&& \quad \times
\frac{(q^{-N})_{|\gamma|}}
{(d)_{|\gamma|}}
\ = \
\prod_{ 1 \le i \le n}
\frac
{((d e  / b c) z_i, (e/a_i) x_i )_N}
{((d e/ a_i b c) z_i, e x_i )_N}
\sum_{\delta \in {\Bbb N}^{n}}
q^{|\delta|} 
\frac{\Delta({z} q^{\delta})}{\Delta({z})}
\nonumber 
\\
&& \quad \quad \quad 
\times
\frac
{(q^{-N})_{|\delta|}}
{(d)_{|\delta|}}
\prod_{1 \le i, j \le n}
\frac{(a_j z_i / z_j)_{\delta_i}} {(q z_i / z_j)_{\delta_i}}
\prod_{1 \le i \le n}
\frac{((d/b) z_i, (d/c) z_i )_{\delta_i}} 
{((d e / b c ) z_i, (A q^{1-N}/ e) z_i)_{\delta_i}}
,
\nonumber
\end{eqnarray}
where $z_i = a_i / A x_i$ for $i= 1, 2, \cdots , n$.
\end{prop}

\begin{proof}
We use \eqref{m1IDT1} and \eqref{rn1DT1}. Note that both of the series in the right hand side of \eqref{m1IDT1} 
and in the left hand side of \eqref{rn1DT1} are very-well-poised-balanced ${}_{2n+6} W_{2n+5}$ 
series.
In this case, we consider the following change of variables:
\begin{eqnarray} \label{covST3}
\tilde{a} = \frac{d e^2 q^{-1}}{A b c},  & \tilde{b} = e/b, &
\tilde{d} = e/c,
\\
\tilde{u}_i = (d e q^{-1} /Ab c)  x_i^{-1}, & \tilde{v}_i = (e/ a_i) x_i  \quad \quad 
(i= 1, \cdots , n), &
\nonumber
\end{eqnarray}
which is a transposition of the variables in 
${}_{2n + 6} W_{2n+5} $ series in \eqref{m1IDT1}.
Since the series in the right hand side of \eqref{m1IDT1} is invariant under this change of 
variables. This invariance implies the desired result \eqref{ltAnST3} by applying 
\eqref{rn1DT1} according to the change \eqref{covST3}.
\end{proof}

\medskip
\noindent
{\bf Rectangular version}

\begin{cor}
\begin{eqnarray}\label{itAnST3}
&&
\sum_{\gamma \in {\Bbb N}^{n}}
q^{|\gamma|} 
\frac{\Delta({x} q^{\gamma})}{\Delta({x})}
\prod_{1 \le i, j \le n}
\frac{(q^{-m_j} x_i / x_j)_{\gamma_i}} {(q x_i / x_j)_{\gamma_i}}
\prod_{1 \le i \le n}
\frac{(b x_i, c x_i )_{\gamma_i}} 
{(e x_i,
 (a b c q^{1-|M|}/ d e) x_i)_{\gamma_i}}
\\
&&
\quad 
\times
\frac{(a)_{|\gamma|}}
{(d)_{|\gamma|}}
\ = \
\prod_{ 1 \le i \le n}
\frac
{((d e  / b c) z_i, (e/a) x_i )_{m_i}}
{((d e/ a b c) z_i, e x_i )_{m_i}}
\
\sum_{\delta \in {\Bbb N}^{n}}
q^{|\delta|} 
\frac{\Delta({z} q^{\delta})}{\Delta({z})}
\nonumber \\
&& \quad \quad \quad 
\times
\frac
{(a)_{|\delta|}}
{(d)_{|\delta|}}
\prod_{1 \le i, j \le n}
\frac{(q^{-m_j} z_i / z_j)_{\delta_i}} {(q z_i / z_j)_{\delta_i}}
\prod_{1 \le i \le n}
\frac{((d/b) z_i, (d/c) z_i )_{\delta_i}} 
{((d e/ b c) z_i, (a q^{1-|M|}/ e) z_i)_{\delta_i}},
\nonumber
\end{eqnarray}
where $z_i = q^{m_i - |M|} x_i^{-1}$ for $i= 1, 2, \cdots , n$.
\end{cor}

\begin{rem}
\eqref{itAnST3} has originally appeared as Theorem 6.8 in  Milne-Lilly \cite{LM1}.
Though they referred as $C_r$ Sears transformation formula there, the sums 
in both side of \eqref{itAnST3} are $A_n$ ${}_4 \phi_3$ series.
For their point of view, it may be precise to refer it as { \it $A_n$ Sears transformation formula
arising from $C_n$ Bailey transform}.

In the case when $n=1$ and $x_1 = 1$, \eqref{ltAnST3} and \eqref{itAnST3} reduce to the Sears transformation
\eqref{SearsT1}.
We also note that \eqref{ltAnST3} can be obtained from $A_n$ Bailey transformation 
\eqref{Masatoshi-san} in the following
way: First we replace $d \to aq/ d$ and $f \to aq/f$ in \eqref{Masatoshi-san}. 
Then let the parameter $a$ tends to infinity in the resulting equation. 
Finally rearrange the parameters appropriately. 

\end{rem}

\subsection{Nonterminating ${}_8 W_7$ transformations}

Here we show a $A_n$ nonterminating ${}_8 W_7$  
transformation formula. Our tool to produce it is $m=1$ case \eqref{m1NT87} of 
of the nonterminating ${}_8 W_7 $ transformation formula \eqref{NT87} 
in Section 3.3. 
One can see the way to prove the identity as

\vspace{5mm}
\begin{center}
\setlength{\unitlength}{1mm}
\begin{picture}(100, 30)
\put(0,20){\makebox(15,10)[r]{${}_8 W_7$ series in $A_n$}}
\put(80,20){\makebox(15,10)[r]{${}_8 W_7$ series in $A_n$}}
\put(20, 25){\vector(1,0){45}}
\put(23, 27){$A_n$ nonterminating ${}_8 W_7$  trans.}
\put(0,0){\makebox(15,10)[r]{${}_{2n+6} W_{2n+5}$ series}}
\put(80,0){\makebox(15,10)[r]{${}_{2n+6} W_{2n+5}$ series}}
\put(20, 5){\vector(1,0){45}}
\put(30, 7){transposition}
\put(5, 20){\vector(0, -1){10}}
\put(75, 10){\vector(0, 1){10}}
\put(8, 15){\eqref{m1NT87}}
\put(78, 15){\eqref{m1NT87}}
\end{picture}
\end{center}

\begin{prop}
\begin{eqnarray}\label{AnNT87T1}
&&
\sum_{\gamma \in {\Bbb N}^{n}}
x_1^{- \gamma_1}  \cdots x_n^{- \gamma_n} 
\left( \frac{a^2 q^2}{b c d E f}\right)^{|\gamma|}
q^{ e_2 ( \gamma )} \  
\frac{\Delta({x} q^{\gamma})}{\Delta({x})}
\prod_{1 \le i \le n}
\frac{1-  a q^{|\gamma| + \gamma_i} x_i}
{1-  a x_i }
\\
&&
\quad 
\times
\prod_{1 \le i, j \le n}
\frac{(e_j x_i / x_j)_{\gamma_i}} {(q x_i / x_j)_{\gamma_i}}
\prod_{1 \le i\le n}
\frac{(b x_i, c x_i, d x_i)_{\gamma_i}} 
{( ( a q / f ) x_i)_{\gamma_i}}
\nonumber\\
&&
\quad \quad  
\times
\frac{( f)_{|\gamma|}}
{(  a q / b, a q / c, a q / d)_{|\gamma|}}
\prod_{1 \le i \le n}
\frac{(a x_i)_{|\gamma|}}
{( ( a q /e_i ) x_i)_{|\gamma|}}
\nonumber\\
&& 
=
\prod_{ 1 \le i \le n}
\frac{
( a q x_i, (a q / e_i f) x_i, 
(\lambda q / e_i ) z_i,
(\lambda q / f ) z_i)_\infty}
{( ( a q / e_i) x_i, (a q / f) x_i, 
(\lambda q / e_i f ) z_i,
\lambda q  z_i)_\infty}
\nonumber\\
&&
\quad 
\times
\sum_{\delta \in {\Bbb N}^{n}}
z_1^{- \delta_1} \cdots z_n^{- \delta_n}
\left( \frac{ a q }{E f}\right)^{|\delta|}
q^{ e_2 ( \delta )} \ 
\frac{\Delta({z} q^{\delta})}{\Delta({z})}
\prod_{1 \le i \le n}
\frac{1-  \lambda q^{|\delta| + \delta_i} z_i}
{1-  \lambda z_i}
\nonumber\\
&&
\quad \quad 
\times
\prod_{1 \le i, j \le n}
\frac{(e_j z_i / z_j)_{\delta_i}} {(q z_i / z_j)_{\delta_i}}
\prod_{1 \le i\le n}
\frac{( ( a q / c d ) z_i, ( a q / b d ) z_i, 
(a q / b c)  z_i)_{\delta_i}} 
{( ( \lambda q / f ) z_i)_{\delta_i}}
\nonumber\\
&&
\quad \quad \quad 
\times
\frac
{( f)_{|\delta|}}
{( a q / b, a q / c,  a q / d )_{|\delta|}}
\prod_{1 \le i \le n}
\frac{(\lambda z_i)_{|\delta|}}
{( ( \lambda q /e_i ) z_i)_{|\delta|}}
\nonumber
\end{eqnarray}
where $ \lambda = a^{2} q^{} / b  c d $
and $z_i = \displaystyle{\frac{e_i}{E} x_i^{-1}}$.
\end{prop}

\begin{proof}
We iterate \eqref{m1NT87}
\begin{eqnarray*}
&&
\sum_{\gamma \in {\Bbb N}^{n}}
x_i^{-\gamma_i} \cdots x_n^{-\gamma_n}
 \left(
\frac{\mu f}{a}
\right)^{|\gamma|}
q^{ e_2 ( \gamma )} \ 
\frac{\Delta({x}q^{\gamma})}
{\Delta({x})}
\prod_{1 \le i \le n}
\frac{1-  a q^{|\gamma| + \gamma_i} x_i }
{1-  a x_i}
\\
&& \quad \quad
\times
\prod_{1 \le i, j \le n}
\frac{(b_j x_i / x_j)_{\gamma_i}} {(q x_i / x_j)_{\gamma_i}}
\prod_{1 \le i \le n}
\frac{(c x_i,  d x_i, e x_i)_{\gamma_i}}
{(( a q / f ) x_i)_{\gamma_i}}
\nonumber\\
&& \quad \quad \quad \quad
\times
\frac{( f)_{|\gamma|}}
{( a q / c, a q / d, a q / e)_{|\gamma|}}
\prod_{1 \le i \le n}
\frac{(a x_i)_{|\gamma|}}
{( ( a q /b_i ) x_i)_{|\gamma|}}
\nonumber\\
&=&
\frac
{(\mu c f / a, \mu d f / a, \mu e f / a, f)_\infty}
{(a q / c, a q / d, a q / e, \mu q)_\infty}
\prod_{ 1 \le i \le n}
\frac
{(a q x_i,  ( \mu b_i f / a ) x_i^{-1} )_\infty}
{((a q / b_i) x_i, (\mu f / a) x_i^{-1} )_\infty}
\nonumber\\
&&
\quad \quad
\times
{}_{2 n + 4} W_{2 n + 3}
\left[
\mu; \{(a q /b_i^{} f ) x_i \}_{n}
a q / c f,a q / d f,a q / e f,
\{(\mu f / a )  x_i^{-1} \}_{n}
; q; f
\right],\nonumber \\
&& \quad \quad \quad \quad \quad \quad \quad 
(\mu = a^{3} q^{2} / B c d e f^{2}),
\end{eqnarray*}
twice. 
On the way to obtain \eqref{AnNT87T1}, we  interchange 
$ (aq / b_i f) x_i$ and $(\mu f / a) x_i^{-1}$ for 
all $i=1, \cdots , n$ simultaneously in the 
${}_{2n+6} W_{2n+5}$ series.
\end{proof}

\begin{rem}
In the case when $n=1$ and $x_1 = 1$, \eqref{AnNT87T1} reduces to the following
nonterminating ${}_8 W_7$
transformation:
\begin{eqnarray}\label{NonTermT87}
&& \quad \quad \quad
{}_{8} W_{7} 
\left[
		a;b, c, d, e, f
		; q; \frac{a^2 q^2}{b c d e f}
\right]\\
&=&
\frac{(a q, a q /e f, \lambda q /e, \lambda q /f)_\infty}
{(a q/ e, a q /f, \lambda q, \lambda q /e f)_\infty }
{}_{8} W_{7} 
\left[
		\lambda; \lambda b /a,
		\lambda c/ a, \lambda d / a , e, f 
		; q; \frac{a q}{e f}
\right],\nonumber
\end{eqnarray}
where $\lambda = a^2 q^{} /b c d$. 
\eqref{AnNT87T1} can also be obtained by taking the limit $N \to \infty$ 
in $A_n$ Bailey transformation formula \eqref{Masatoshi-san}.
\end{rem}


\subsection{Terminating ${}_8 W_7$ transformations}

Here we present $A_n$ nonterminating ${}_8 W_7$  
transformation formulas. We give a proof  by using $m=1$ case \eqref{m1T87T1} of 
of the nonterminating ${}_8 W_7 $ transformation formula \eqref{T87T1} 
in Section 3.3.  The proof is in the same manner as in that of \eqref{AnNT87T1}. 
One can see the way to give the identity as

\vspace{5mm}
\begin{center}
\setlength{\unitlength}{1mm}
\begin{picture}(100, 30)
\put(0,20){\makebox(15,10)[r]{${}_8 W_7$ series in $A_n$}}
\put(80,20){\makebox(15,10)[r]{${}_8 W_7$ series in $A_n$}}
\put(20, 25){\vector(1,0){45}}
\put(23, 27){$A_n$ terminating ${}_8 W_7$  trans.}
\put(0,0){\makebox(15,10)[r]{${}_{2n+6} W_{2n+5}$ series}}
\put(80,0){\makebox(15,10)[r]{${}_{2n+6} W_{2n+5}$ series}}
\put(20, 5){\vector(1,0){45}}
\put(30, 7){transposition}
\put(5, 20){\vector(0, -1){10}}
\put(75, 10){\vector(0, 1){10}}
\put(8, 15){\eqref{m1T87T1}}
\put(78, 15){\eqref{m1T87T1}}
\end{picture}
\end{center}

\begin{prop}
\begin{eqnarray}\label{ltAnTerm8W7-1}
&&
\quad 
\sum_{\gamma \in {\Bbb N}^{n}}
x_1^{\gamma_1} \cdots x_n^{\gamma_n}  
\left( \frac{a^2 q^{N+2}}{B c f g} \right)^{|\gamma|} 
q^{ - e_2 ( \gamma )} \
\frac{\Delta({x} q^{\gamma})}{\Delta({x})}
\prod_{1 \le i \le n}
\frac{1-  a q^{|\gamma| + \gamma_i} x_i}
{1-  a x_i}
\\
&& \quad \quad \quad 
\times
\prod_{1 \le i \le n}
\frac{(a x_i)_{|\gamma|}}
{( ( a q /b_i ) x_i)_{|\gamma|}}
\prod_{1 \le i, j \le n}
\frac{(b_j x_i / x_j)_{\gamma_i}} {(q x_i / x_j)_{\gamma_i}}
\nonumber\\
&& \quad \quad \quad \quad \quad 
\times
\frac{( f, g, q^{-N})_{|\gamma|}}
{(  a q / c )_{|\gamma|}}
\prod_{1 \le i\le n}
\frac{(c x_i)_{\gamma_i}} 
{( a q^{N+1} x_i, ( a q / f ) x_i, (a q/ g) x_i)_{\gamma_i}}
\nonumber\\
&&  
= \
\prod_{ 1 \le i \le n}
\frac{
( a q x_i, (a q / b_i f) x_i, (a q / b_i g) x_i, (a q / f g) x_i)_N}
{( ( a q / b_i) x_i, (a q / f) x_i, (a q / g ) x_i, ( a q / b_i f g) x_i)_N}
\nonumber\\
&&
\quad \quad 
\times
\sum_{\delta \in {\Bbb N}^{n}}
z_1^{\delta_1} \cdots z_n^{\delta_n} 
\left( \frac{q}{c} \right)^{|\delta|} 
q^{ - e_2 ( \delta )} \ 
\frac{\Delta({z} q^{\delta})}{\Delta({z})}
\prod_{1 \le i \le n}
\frac{1-  ( q^{-N-1} B f g / a) q^{|\delta| + \delta_i} z_i }
{1-  (q^{-N-1} B f g / a) z_i}
\nonumber\\
&&
\quad \quad \quad \quad
\times
\prod_{1 \le i \le n}
\frac{( (q^{-N} B f g / a) z_i)_{|\delta|}}
{( ( q^{-N}  f g / a  ) z_i )_{|\delta|}}
\prod_{1 \le i, j \le n}
\frac{(b_j z_i / z_j)_{\delta_i}} {(q z_i / z_j)_{\delta_i}}
\nonumber\\
&& 
\quad \quad \quad \quad \quad \quad
\times
\frac{( f, g, q^{-N})_{|\delta|}}
{( a q / c)_{|\delta|}}
\prod_{1 \le i\le n}
\frac{( (  q^{-N-1} B c f g / a^2 ) z_i )_{\delta_i}} 
{( ( B f g /a ) z_i, ( q^{-N}  B g / a ) z_i, ( q^{-N} B f / a) z_i )_{\delta_i}}
.
\nonumber
\end{eqnarray}
where $z_i = \displaystyle{\frac{e_i}{E} x_i^{-1}}$.
\end{prop}

\medskip
\noindent
{\bf Rectangular version}

\begin{cor}
\begin{eqnarray}\label{itAnTerm8W7-1}
&&
\sum_{\gamma \in {\Bbb N}^{n}}
x_1^{\gamma_1} \cdots x_n^{\gamma_n}  
\left( \frac{a^2 q^{|M|+2}}{b c f g} \right)^{|\gamma|} 
q^{ - e_2 ( \gamma )} \
\frac{\Delta({x} q^{\gamma})}{\Delta({x})}
\prod_{1 \le i \le n}
\frac{1-  a q^{|\gamma| + \gamma_i} x_i}
{1-  a x_i}
\\
&& \quad \quad 
\times
\prod_{1 \le i \le n}
\frac{(a x_i)_{|\gamma|}}
{( a q^{|M|+1 }  x_i)_{|\gamma|}}
\prod_{1 \le i, j \le n}
\frac{(q^{-m_j} x_i / x_j)_{\gamma_i}} {(q x_i / x_j)_{\gamma_i}}
\nonumber\\
&& \quad \quad  \quad \quad 
\times
\frac{( b, f, g)_{|\gamma|}}
{(  a q / c )_{|\gamma|}}
\prod_{1 \le i \le n}
\frac{(c x_i)_{\gamma_i}} 
{( (a q /b) x_i, ( a q / f ) x_i, (a q/ g) x_i)_{\gamma_i}}
\nonumber\\
&=&
\prod_{ 1 \le i \le n}
\frac{
( a q x_i, (a q / b f) x_i, (a q / b g) x_i, (a q / f g) x_i)_{m_i}}
{( ( a q / b) x_i, (a q / f) x_i, (a q / g ) x_i, ( a q / b f g) x_i)_{m_i}}
\nonumber\\
&& \quad \quad 
\times
\sum_{\delta \in {\Bbb N}^{n}}
z_1^{\delta_1} \cdots z_n^{\delta_n} 
\left( \frac{q}{c} \right)^{|\delta|} 
q^{ - e_2 ( \delta )} \ 
\frac{\Delta({z} q^{\delta})}{\Delta({z})}
\prod_{1 \le i \le n}
\frac{1-  ( q^{-|M|-1} b f g / a) q^{|\delta| + \delta_i} z_i }
{1-  (q^{-|M|-1} b f g / a) z_i}
\nonumber\\
&& \quad \quad \quad \quad 
\times
\prod_{1 \le i \le n}
\frac{( (q^{-|M|} b f g / a) z_i)_{|\delta|}}
{( ( q^{-|M|}  f g / a  ) z_i )_{|\delta|}}
\prod_{1 \le i, j \le n}
\frac{( q^{ - m_j} z_i / z_j)_{\delta_i}} {(q z_i / z_j)_{\delta_i}}
\nonumber\\
&& \quad \quad \quad \quad \quad \quad 
\times
\frac{( b, f, g)_{|\delta|}}
{( a q / c)_{|\delta|}}
\prod_{1 \le i\le n}
\frac{( (  q^{-|M|-1} b c f g / a^2 ) z_i )_{\delta_i}} 
{( ( q^{- |M|} f g /a ) z_i, ( q^{-|M|}  b g / a ) z_i, ( q^{-|M|} b f / a) z_i )_{\delta_i}}
.
\nonumber
\end{eqnarray}
where $z_i = \displaystyle{q^{m_i - |M|} x_i^{-1}}$.
\end{cor}

\begin{rem}
In the case when $n=1$ and $x_1 = 1$, \eqref{ltAnTerm8W7-1} and \eqref{itAnTerm8W7-1} reduces to the following
terminating ${}_8 W_7$ transformation:
\begin{eqnarray}\label{n1-AnTerm8W7-1}
&&
{}_8 W_7
\left[ 
a; b, c, f, g, q^{-N}; q; 
\frac{a^2 q^{2+N}}{ b c f g}
\right]
=
\frac{
( a q, a q / b f, a q / b g, a q / f g)_N}
{( a q / b, a q / f, a q / g,  a q / b f g)_N}
\\
&&
\quad \quad \quad \quad 
{}_8 W_7
\left[
q^{-N-1} b f g / a;
b, q^{-N-1} b c f g / a^2, f, g, q^{-N}; q; 
\frac{q}{c}
\right].
\nonumber
\end{eqnarray}
\eqref{ltAnTerm8W7-1} can also be obtained by taking the limit $d \to \infty$ in \eqref{Masatoshi-san}.
\end{rem}

\begin{rem}
$A_n$ Watson transformation \eqref{ltAnWT3} can be obtained by combining \eqref{ltAnTerm8W7-1} and 
\eqref{ltAnWT4} and by combining \eqref{ltAnWT4} and $A_n$ Sears transformations formula \eqref{ltAnST3}.
\begin{center}
\setlength{\unitlength}{1mm}
\begin{picture}(100, 30)
\put(0,20){\makebox(15,10)[r]{${}_8 W_7$ series in $A_n$}}
\put(80,20){\makebox(15,10)[r]{${}_8 W_7$ series in $A_n$}}
\put(20, 25){\vector(1,0){45}}
\put(37, 27){\eqref{ltAnTerm8W7-1}}
\put(0,0){\makebox(15,10)[r]{ ${}_4 \phi_3$ series in $A_n$}}
\put(80,0){\makebox(15,10)[r]{${}_{4} \phi_3$ series in $A_n$ }}
\put(20, 5){\vector(1,0){45}}
\put(37, 7){\eqref{ltAnST3}}
\put(5, 20){\vector(0, -1){10}}
\put(75, 20){\vector(0, -1){10}}
\put(6, 15){\eqref{ltAnWT4}}
\put(77, 15){\eqref{ltAnWT4}}
\put(23, 20){\vector(4,-1){40}}
\put(50, 15){\eqref{ltAnWT3}}
\end{picture}
\end{center}

\end{rem}

As we have seen in this section, our discussion here may implies not only that our class of multiple 
hypergeometric transformations in the previous section are broader class than Milne`s class of  transformation in $A_n$
but also that our class contains more precise informations. In our terminology, one may see that
Milne`s hypergeometric transformations have extra hidden symmetries in the one dimensional (generalized)
hypergeometric series.

\vspace{15mm}

\noindent
{\bf {\large Acknowledgments}}
\\
I would like to express my sincere thanks to Professors Etsuro Date and Masatoshi Noumi for 
their encouragements.
 
}


\vspace{25mm}



\end{document}